\documentclass[a4paper,12pt,oneside]{article}				         
\usepackage[left=1.5cm,right=1.5cm,top=3cm,bottom=3cm]{geometry}         

\usepackage[latin1]{inputenc}
\usepackage{amsfonts, amsmath, amssymb, latexsym}
\usepackage{amscd}
\usepackage[usenames,dvipsnames]{color}
\usepackage{epsfig}
\usepackage{algorithmic}
\usepackage{algorithm2e}
\usepackage{enumerate}

\usepackage{tabularx}
\usepackage{array}
\usepackage{fancyvrb}
\usepackage{etoolbox}

\usepackage{tikz} 
\usepackage{stmaryrd}
\usepackage[normalem]{ulem}

\usepackage{graphicx}
\newcommand{\hookuparrow}{\mathrel{\rotatebox[origin=c]{90}{$\hookrightarrow$}}}

\usepackage{caption}

\usepackage[backref]{hyperref}
\usepackage[alphabetic,backrefs,lite]{amsrefs}

\usepackage[all]{xy}

\usepackage{epstopdf}
\DeclareGraphicsExtensions{.pdf,.eps,.png,.jpg,.mps}

\usepackage{rotating}

\usepackage[usenames,dvipsnames]{color}

\usepackage{amsthm}
\newtheorem*{teo*}{Theorem}
\newtheorem*{conj*}{Conjecture}
\newtheorem{lemma}{Lemma}[section]
\newtheorem{teo}[lemma]{Theorem}
\newtheorem{prop}[lemma]{Proposition}
\newtheorem{cor}[lemma]{Corollary}

\theoremstyle{definition}
\newtheorem{defn}[lemma]{Definition}
\newtheorem{nota}[lemma]{Notation}
\newtheorem{remark}[lemma]{Remark}
\newtheorem{remarks}[lemma]{Remarks}

\newtheorem{notation}[lemma]{Notation}

\begin{document}
\title{Mumford curves covering $p$-adic Shimura curves and their fundamental domains}
\author{Laia Amor\'{o}s, Piermarco Milione}

\newcommand{\Addresses}{{
  \bigskip
  \footnotesize

  L.~Amorós, \textsc{Facult\'{e} des Sciences, de la Technologie et de la Communication. 6 rue Richard Coudenhove-Kalergi, L-1359 Luxembourg}\par\nopagebreak
  \textit{E-mail address} \texttt{ laia.amoros@uni.lu}

  \medskip

  P.~Milione , \textsc{School of Science. Department of Mathematics and Systems Analysis. Aalto University. FI-00076 Aalto. Finland}\par\nopagebreak
  \textit{E-mail address} \texttt{piermarco.milione@aalto.fi}

}}

\date{}


\maketitle

\abstract{
\small
We give an explicit description of fundamental domains associated with the $p$-adic uniformisation of families of Shimura curves of discriminant $Dp$ and level $N\geq 1$, for which the one-sided ideal class number $h(D,N)$ is $1$. The results obtained generalise those in \cite[Ch. IX]{Gerritzen_vanderPut1980} for Shimura curves of discriminant $2p$ and level $N=1$. The method we present here enables us to find Mumford curves covering Shimura curves, together with a free system of generators for the associated Schottky groups, $p$-adic good fundamental domains, and their stable reduction-graphs. The method is based on a detailed study of the modular arithmetic of an Eichler order of level $N$ inside the definite quaternion algebra of discriminant $D$, for which we generalise the classical results of Hurwitz \cite{Hurwitz1896}. As an application, we prove general formulas for the reduction-graphs with lengths at $p$ of the families of Shimura curves considered.
}

\smallskip
\small
\textbf{Mathematics Subject Classification (2010)} 11G18, 11R52, 14G35, 14G22.
\smallskip

\small
\textbf{Keywords} Shimura curves, Mumford curves, $p$-adic fundamental domains. 

\bibstyle{plain}

\section*{Introduction}
Shimura curves are a wide and powerful generalisation of modular curves, as shown by many applications that have solved important problems in number theory in the last few decades. The complex uniformisation of Shimura curves, in comparison to that of modular curves, is more difficult to approach computationally (think about the lack of cusps, which complicates the computation of Fourier expansions of their uniformising functions). Nevertheless, several studies make amenable computations for this complex uniformisation such as \cite{AlsinaBayer2004}, \cite{Voight2006}, \cite{BayerTravesa2007} and \cite{VoightWillis2011}. 
In addition to the complex uniformisation, Shimura curves also admit non-archimedean uniformisations. As expressed by the fundamental theorems of \v{C}erednik (\cite{Cerednik1976B}) and Drinfel'd (\cite{Drinfeld1976}). The non-archimedean uniformisation describes a $p$-adic integral model of Shimura curves associated with indefinite quaternion algebras of discriminant $Dp$ and its bad special fibre, thanks to the unified language of rigid analytic geometry.

Several studies take a computational approach to the study of $p$-adic uniformisations of Shimura curves and their bad reduction fibres. In particular in \cite{FrancMasdeu2014} the problem of computing the reduction-graph with lengths of a Shimura curve associated with an arbitrary Eichler order is successfully solved by giving an algorithm implemented in Sage. Other studies taking an effective study of the $p$-adic uniformisation of Shimura curves through the computation of their special points are, for example, \cite{DarmonPollack}, \cite{Greenberg2006} and \cite{Greenberg2009}.

In this paper we present some results on these $p$-adic uniformisations which are of a certain generality, since they hold for some infinite families of Shimura curves, and explicit, because they are approachable from a computational point of view. To be more specific, one of the advantages we point out in the $p$-adic approach to uniformisation of Shimura curves of discriminant $Dp$ is that the ``trick'' of \emph{interchanging the local invariants} $p$ and $\infty$ of the indefinite quaternion algebra leads to the study of the definite quaternion algebra $H$ of discriminant $D$, which does not depend on $p$ (at least at first glance). This gives rise to infinitely many families of Shimura curves $X(Dp,N)$, indexed by $p$, and such that each family depends on the arithmetic behaviour of an Eichler order of level $N$ inside $H$. 

We bring to light two classical studies: one of Hurwitz \cite{Hurwitz1896} and one of Gerritzen and van der Put \cite{Gerritzen_vanderPut1980}. In \cite[Ch. IX]{Gerritzen_vanderPut1980} they study Mumford curves covering Shimura curves of discriminant $2p$ and level $1$. In \cite{Hurwitz1896}, the author studies the arithmetic of a maximal order inside the definite quaternion algebra of discriminant $2$. In particular, he proves a unicity result in this order which we are able to generalise in our situation. The combination and generalisation of these two studies led us to study the $p$-adic uniformisation of families of Shimura curves $X(Dp,N)$ such that the one-sided ideal class number $h(D,N)$ is $1$, through the study of certain Mumford curves covering them. The fundamental results of this paper (Theorem \ref{Theorem_Schottky_group} and Corollary \ref{Corollary}) can be summarised in the following statement.

\begin{teo*}
Let $X(Dp,N)$ be the Shimura curve associated with an Eichler order of level $N$ with $(N,Dp) = 1$, inside the indefinite quaternion algebra of discriminant $Dp$, and let $\mathcal{O}$ be an Eichler order of level $N$ inside the definite quaternion algebra of discriminant $D$. Assume that 
\begin{enumerate}[$(i)$]
\item
$h(D,N)=1$.
\item
There exists $\xi\in\mathcal{O}$ such that $2\in\xi\mathcal{O}$ and
$$\varphi:\mathcal{O}^{\times}/\mathbb{Z}^{\times}\rightarrow (\mathcal{O}/\xi\mathcal{O})_r^{\times},\quad\varphi(u)=\lambda_r(u+\mathrm{Nm}(\xi)\mathcal{O})$$
is a bijection, where $(\mathcal{O}/\xi\mathcal{O})_r^{\times}$ denotes the image of $(\mathcal{O}/\mathrm{Nm}(\xi)\mathcal{O})^{\times}$ under the natural projection $\lambda_r:\mathcal{O}/\mathrm{Nm}(\xi)\mathcal{O}\twoheadrightarrow\mathcal{O}/\xi\mathcal{O}$. 
\item
The prime $p$ satisfies
$
t(p):=\#\{\alpha\in\mathcal{O}\mid\,\mathrm{Nm}(\alpha)=p,\,\alpha-1\in\xi\mathcal{O},\,\mathrm{Tr}(\alpha)=0\}=0.
$

\end{enumerate}
Then there exists a Mumford curve $C$ covering the Shimura curve $X(Dp,N)\otimes_{\mathbb{Q}}\mathbb{Q}_{p}$ such that:
\begin{enumerate}[$(a)$]
\item
the curve $C$ has genus $(p+1)/2$,
\item
the degree of the cover is $\#\mathcal{O}^{\times}/\mathbb{Z}^{\times}$,
 \item
a finite set of generators for the Schottky group uniformising the Mumford curve $C$ is given by the image, inside $\mathrm{PGL}_{2}(\mathbb{Q}_{p})$, of the set of matrices 
\[
\widetilde{S}:=\Phi_{p}(\{\alpha\in\mathcal{O}\mid\,\mathrm{Nm}(\alpha)=p,\,\alpha-1\in\xi\mathcal{O}\})\subseteq\mathrm{GL}_{2}(\mathbb{Q}_{p}),
\]
where $\Phi_{p}$ denotes a matrix immersion of the definite quaternion algebra inside $\mathrm{M}_{2}(\mathbb{Q}_{p})$.
\end{enumerate}
\end{teo*}

	In Table \ref{Table_xi} we present an element satisfying condition ($ii$) for all definite quaternion algebras with $h(D,N)=1$, except for the cases $(D,N)=(2,5)$ and $(7,1)$, for which no such element exists. In section \ref{null-trace_condition} we show that there are infinitely many primes satisfying ($ii$).

Interestingly point ($a$) of the theorem is a consequence of a result that we prove concerning the number of representations of the prime $p$ by a quaternary quadratic form with some congruence conditions on the coefficients. 

\begin{teo*}
Let $\xi\in\mathcal{O}$ be a quaternion that satisfies condition $(ii)$ of the previous Theorem. Then the finite set
$
\{\alpha\in\mathcal{O}\mid\,\mathrm{Nm}(\alpha)=p,\,\alpha-1\in\xi\mathcal{O}\}
$
has cardinality $2(p+1)$.
\end{teo*}

	This article is organised as follows. 
	In section \ref{Sec1} we introduce the $p$-adic upper half-plane and the Bruhat-Tits tree together with the reduction map identifying it with the reduction of this rigid analytic variety. Moreover we recall the theory of Mumford curves and we relate it to that of the $p$-adic uniformisation of Shimura curves. 
	In section \ref{Sec2} we introduce and solve the problem of unique factorisation in definite quaternion orders with one-sided ideal class number equal to $1$. 
	In section \ref{Sec3} we apply the arithmetic results obtained in Section \ref{Sec2} in order to prove the main results of the paper. We also describe good fundamental domains and stable reduction-graphs for the Mumford curves covering our Shimura curves. 
	In section \ref{Sec4} we show how our detailed knowledge of the $p$-adic Shimura curve as a rigid analytic variety allows us to easily obtain formulas describing the reduction-graphs with lengths of these Shimura curves as well as genus formulas for them and for certain Atkin-Lehner quotients.
	Finally in section \ref{Sec5} we show, in a specific example, how one can use the main results in order to compute free systems of generators for the Schottky group predicted by the theorem above, a $p$-adic good fundamental domain on $\mathcal{H}_p$ for the action of this Schottky group and its stable reduction-graph, as well as the reduction-graph with lengths of the Shimura curve considered. This is done thanks to an algorithm developed in the computer algebra system Magma \cite{Magma}.

\subsection*{Acknowledgements} 
The authors wish to thank Pilar Bayer and Gabor Wiese for their valuable guidance and support given before and during the preparation of this paper, and for their careful reading of this manuscript. The authors also aknowledge several helpful and enlightening conversations with Marc Masdeu and Xavier Xarles. 

\subsection*{Notation}
Throughout the paper, $p$ will denote a fixed odd prime integer and $\overline{\mathbb{Q}}_{p}$ a fixed algebraic closure of the field of $p$-adic numbers $\mathbb{Q}_{p}$. Once this is fixed, $\mathbb{Q}_{p^{2}}\subseteq\overline{\mathbb{Q}}_{p}$ will denote the unramified quadratic extension of $\mathbb{Q}_{p}$ and $\mathbb{Z}_{p^{2}}$ its ring of integers, $\mathbb{Q}_{p}^{nr}\subseteq\overline{\mathbb{Q}}_{p}$ the maximal unramified extension of $\mathbb{Q}_{p}$ and $\mathbb{Z}_{p}^{nr}$ its ring of integers.

We take the usual convention for the $p$-adic absolute value to be defined as $|z|:=1/p^{v_{p}(z)}$ for every $z\in\mathbb{Q}_{p}$, where $v_{p}$ denotes the $p$-adic valuation on $\mathbb{Q}_{p}$. Finally $\mathbb{C}_{p}$ will denote the completion of $\overline{\mathbb{Q}}_{p}$ with respect to the unique absolute value extending the $p$-adic absolute value.

\section{$p$-adic uniformisation of Shimura curves}\label{Sec1}

\subsection{The $p$-adic upper half-plane}
Let $\mathbb{P}^{1,rig}:=\mathbb{P}^{1,rig}_{\mathbb{Q}_{p}}$ denote the rigid analytic projective line over $\mathbb{Q}_{p}$. This is the rigidification of the algebraic projective line $\mathbb{P}^{1}_{\mathbb{Q}_{p}}$ (cf. \cite[9.3.4]{Bosch1984}), so that its set of $L$-points is $\mathbb{P}^{1,rig}(L)=\mathbb{P}^{1}(L)$, for every extension $\mathbb{Q}_{p}\subseteq L\subseteq\mathbb{C}_{p}$.
For $a=[a_{0}:a_{1}]\in\mathbb{P}^{1}(\mathbb{Q}_{p})$ and $r\in\mathbb{R}_{\geq0}$, we denote by \\
$
\mathbb{B}^{+}(a,|p|^r) :=\{[z_{0}:z_{1}]\in\mathbb{P}^{1}(\mathbb{C}_{p}):\; |z_{0}a_{1}-z_{1}a_{0}| \leq | p |^{r}\}=
\{z\in\mathbb{P}^{1}(\mathbb{C}_{p}):\;v(z_{0}a_{1}-z_{1}a_{0})\geq r\},
$
and
$
\mathbb{B}^{-}(a,|p|^r):=\{[z_{0}:z_{1}]\in\mathbb{P}^{1}(\mathbb{C}_{p}):\; |z_{0}a_{1}-z_{1}a_{0}|<| p |^{r}\}=
\{z\in\mathbb{P}^{1}(\mathbb{C}_{p}):\;v(z_{0}a_{1}-z_{1}a_{0})> r\},
$
the closed ball and open ball of $\mathbb{P}^{1,rig}$ with centre $a$ and radius $| p |^{r}$. 
It can be proved that a complement of open balls in $\mathbb{P}^{1}(\mathbb{C}_{p})$ is an admissible open subset of $\mathbb{P}^{1,rig}$, and that every admissible open subset of $\mathbb{P}^{1,rig}$ is obtained as a union of such open subsets (\cite[9.7.2/2]{Bosch1984}).

\begin{defn}
Consider a functor from the category of field extensions $\mathbb{Q}_{p}\subseteq L\subseteq\mathbb{C}_{p}$ to the category of sets defined as follows: for each extension $L|\mathbb{Q}_{p}$ consider the subset of points
\[
\mathcal{H}_{p}(L):=\mathbb{P}^{1}(L)\smallsetminus\mathbb{P}^{1}(\mathbb{Q}_{p})\subset \mathbb{P}^{1,rig}(L).
\]
$\mathcal{H}_{p}(L)$ is actually the set of $L$-points of a rigid analytic variety $\mathcal{H}_{p}$ over $\mathbb{Q}_{p}$ called the \emph{$p$-adic upper half-plane} over $\mathbb{Q}_{p}$. One way to see this is by defining an open admissible cover of $\mathcal{H}_{p}(L)$. 
\end{defn}

\begin{defn}\label{abiertos_del_semiplano}
For every integer $i\geq 0$ we define the following subsets of points of $\mathbb{P}^{1,rig}(L)$:
\[
\mathcal{H}_{p}^{(i)}(L):=\mathbb{P}^{1}(L)\smallsetminus\bigcup_{a\in\mathbb{P}^{1}(\mathbb{Q}_{p})}\mathbb{B}^{-}(a,|p|^{i}).
\]
\end{defn}

For every point $a=[a_{0}:a_{1}]\in\mathbb{P}^{1}(\mathbb{Q}_{p})$, one can choose \emph{unimodular coordinates}, i.e. $(a_{0},a_{1})\in\mathbb{Z}_{p}^{2}$ such that at least one coordinate is a unit, so it makes sense to consider $a$ mod $p^i$.
The next lemma allows us to simplify the description of the subsets in Definition \ref{abiertos_del_semiplano}.

\begin{lemma}\label{cns_bolas_disjuntas}
Let $i>0$ be an integer and take $a,a'\in\mathbb{P}^{1}(\mathbb{Q}_{p})$. Then 
\begin{enumerate}[$(a)$]
\item
$\mathbb{B}^{+}(a,|p|^{i})\cap\mathbb{B}^{+}(a',|p|^{i})\neq\emptyset\Leftrightarrow a\equiv a'\,(\mathrm{mod}\,p^{i})$.
\item
$\mathbb{B}^{-}(a,|p|^{i-1})\cap\mathbb{B}^{-}(a',|p|^{i-1})\neq\emptyset\Leftrightarrow a\equiv a'\,(\mathrm{mod}\,p^{i})$. \qed
\end{enumerate}
\end{lemma}

	Denote by $\mathcal{P}_{i}$ a system of representatives for the points of $\mathbb{P}^{1}(\mathbb{Q}_{p})=\mathbb{P}^{1}(\mathbb{Z}_{p})$ mod $p^{i}$, i.e. $\mathcal{P}_{i}$ is a system of representatives for the points $\mathbb{P}^{1}\left(\mathbb{Z}_{p}/p^{i}\mathbb{Z}_{p}\right)$ of the projective line. A point $a\in\mathcal{P}_{i}$ has unimodular coordinates $(a_{0},a_{1})$ such that their reductions are $(\tilde{a}_{0},\tilde{a}_{1})\in\mathbb{Z}_{p}/p^{i}\mathbb{Z}_{p}\times\mathbb{Z}_{p}/p^{i}\mathbb{Z}_{p}$ not both $\equiv 0\,(\mathrm{mod}\,p)$. Finally there is a bijection
$
\mathcal{P}_{i}\simeq \mathbb{Z}_{p}/p^{i}\mathbb{Z}_{p}\sqcup p\mathbb{Z}_{p}/p^{i}\mathbb{Z}_{p}.
$ 
In particular $\mathcal{P}_{1}\simeq \mathbb{P}^{1}(\mathbb{F}_{p})= \mathbb{F}_{p}\cup\{\infty\}$. Thus, by that basic property of non-archimedean balls to be either disjoint or one contained inside the other, and after Lemma \ref{cns_bolas_disjuntas}, for every $i>0$,
\[
\mathcal{H}_{p}^{(i)}(L)=\mathbb{P}^{1}(L)\smallsetminus\bigcup_{a\in\mathcal{P}_{i}}\mathbb{B}^{-}(a,|p|^{i-1}).
\]
As a consequence we obtain that $\mathcal{H}_{p}^{(i)}$ are admissible affinoid subdomains of $\mathbb{P}^{1,rig}$.
If one proves that the cover $\{\mathcal{H}_{p}^{(i)}\}_{i>0}$ is admissible, then $\mathcal{H}_{p}$ is proved to be a rigid analytic variety over $\mathbb{Q}_{p}$ (following definitions of \cite[Ch. 9.3]{Bosch1984}). We refer the reader to \cite{SchneiderStuhler1991} for a proof of this.

\subsection{The Bruhat-Tits tree associated with $\mathrm{PGL}_{2}(\mathbb{Q}_{p})$}

	A lattice $M\subseteq \mathbb{Q}_{p}^{2}$ is a free $\mathbb{Z}_{p}$-module of rank $2$. 
	Two lattices $M,M'\subseteq \mathbb{Q}_{p}^{2}$ are said to be homothetic if there exists $\lambda\in \mathbb{Q}_{p}^{\times}$ such that $M'=\lambda M$. We will denote by $\{M\}$ the homothety class of $M$. 
	Given $\{M\},\{M^{'}\}$ we can always choose their representatives such that $p^{n}M\subseteq M'\subseteq M$, for some $n\in\mathbb{N}$ (cf. \cite[1.1]{Serre1977Arbres}). 
	We say that two homothety classes $\{M\},\{M'\}$  are \emph{adjacent} if their representatives can be chosen so that $pM\subsetneq M'\subsetneq M$.


\begin{defn} 
	We define the graph $\mathcal{T}_{p}$ whose set of vertices $\mathrm{Ver}(\mathcal{T}_{p})$ consists of the homothety classes of lattices of $\mathbb{Q}_{p}^{2}$ and whose set of oriented edges $\mathrm{Ed}(\mathcal{T}_{p})$ is the set of pairs of adjacent classes. One can also consider the set of unoriented edges, which is formed by unordered pairs of adjacent classes. The graph $\mathcal{T}_{p}$ is a $(p+1)$-regular tree (cf. \cite[Ch. II]{Serre1977Arbres}) which is known in the literature as the \emph{Bruhat-Tits tree} associated with $\mathrm{PGL}_{2}(\mathbb{Q}_{p})$.
\end{defn}


\begin{nota}\label{oriented_unoriented_edges}
	We denote by $v^{0}\in\mathcal{T}_{p}$ the vertex represented by the lattice $M^{0}:=\langle(1,0),(0,1)\rangle$.
	If $e=(v,v')\in\mathrm{Ed}(\mathcal{T}_{p})$ is an oriented edge then $-e:=(v',v)$ denotes the inverse edge. 
\end{nota}

The tree $\mathcal{T}_{p}$, that we have defined as a combinatorial object can be realised as a topological space by thinking of its vertices as classes of norms over $\mathbb{Q}_{p}$ (cf. \cite[Introduction, Sec. 1]{BoutotCarayol1991}).
With this topology, the map $\gamma\in\mathrm{PGL}_{2}(\mathbb{Q}_{p})\mapsto\gamma\cdot v\in\mathrm{Ver}(\mathcal{T}_{p})$ induces an homeomorphism 
\[
\mathrm{PGL}_{2}(\mathbb{Q}_{p})/\mathrm{PGL}_{2}(\mathbb{Z}_{p})\simeq \mathcal{T}_{p},
\] 
where $\mathrm{PGL}_{2}(\mathbb{Q}_{p})$ is taken with its natural topology.
Therefore, we can represent each vertex by a class of matrices.  Namely, if $v=\{M\}$ then $v$ is represented by the class $\{\alpha_{M}\}\in \mathrm{PGL}_{2}(\mathbb{Q}_{p})/\mathrm{PGL}_{2}(\mathbb{Z}_{p})$ such that $\alpha_{M}$ is the matrix whose columns form a basis of $M$.

\begin{defn}
	With this description in terms of matrices we can define the following ascending chain of subtrees of $\mathcal{T}_{p}$.
For every $i\geq 0$, denote by $\mathcal{T}_{p}^{(i)}$ the subtree of $\mathcal{T}_{p}$ whose set of vertices is 
$$\mathrm{Ver}(\mathcal{T}^{(i)}_{p}):=\{v=\{\alpha\}\mid\;v_{p}(\mathrm{det}\,\alpha)\leq i\}.$$ 
\end{defn}

\begin{remark}
If we look back at the admissible cover defined for the upper half-plane $\mathcal{H}_{p}$, then we observe that the set of representatives $\mathcal{P}_{i}$ for the points of the projective line $\mathbb{P}(\mathbb{Z}_{p}/p^{i}\mathbb{Z}_{p})$ corresponds bijectively with the set of ``added vertices'' of $\mathcal{T}_{p}^{(i)}$, i.e.,  there is a bijection of sets
\[
\mathrm{Ver}(\mathcal{T}^{(i)}_{p})\smallsetminus \mathrm{Ver}(\mathcal{T}^{(i-1)}_{p})\simeq\mathcal{P}_{i},
\]
for every $i\geq 1$.
The intuition then would suggest that removing the open balls in $\mathbb{P}^{1,rig}(\mathbb{C}_{p})$ with centre in $\mathcal{P}_{i}$ corresponds through this bijection to adding to $\mathcal{T}^{(i-1)}_{p}$ the missing vertices of $\mathcal{T}^{(i)}$. This intuition finds its theoretical explanation in the following theorem (cf. \cite{BoutotCarayol1991} and \cite{DasguptaTeitelbaumAWS}).
\end{remark}

\begin{teo}\label{red_map}
For every extension $\mathbb{Q}_{p}\subseteq L\subseteq\mathbb{C}_{p}$, there is a map 
$
\mathrm{Red}:\mathcal{H}_{p}(L)\rightarrow\mathcal{T}_{p}
$
such that
\begin{enumerate}[$(a)$]
\item
It is equivariant with respect to the action of $\mathrm{PGL}_{2}(\mathbb{Q}_{p})$, i.e.  
$\mathrm{Red}(\gamma\cdot z)=\gamma\cdot\mathrm{Red}(z)$
for every $z\in\mathcal{H}_{p}(L)$ and every $\gamma\in\mathrm{PGL}_{2}(\mathbb{Q}_{p})$.
\item
The image $\mathrm{Red}(\mathcal{H}_{p}^{(i)})$ is the rational geometric realisation of $\mathcal{T}_{p}^{(i-1)}$, for every $i\geq 0$. \qed
\end{enumerate}
\end{teo}

	We write this map as $\mathrm{Red}:\mathcal{H}_{p}\rightarrow\mathcal{T}_{p}$ and call it the \emph{reduction map} associated with $\mathcal{H}_{p}$.
	
\begin{remark}
	The reduction map owes its name to its intimate relation with the usual reduction mod $p$.
	Actually, the Tate algebra of series of the affinoid subdomains $\mathcal{H}_{p}^{(i)}$ can be reduced mod $p$ through reducing mod $p$ the restricted series, and this gives an algebraic variety over $\mathbb{F}_{p}$. 
It is an exercise to see that the dual graph of the reduction mod $p$ of $\mathcal{H}_{p}^{(i)}$ 
is the tree $\mathcal{T}_{p}^{(i-1)}$. 
	This is done in \cite[2.3]{BoutotCarayol1991} where, in fact, the affinoid subdomain $\mathcal{H}_{p}^{(i)}$ is defined by $\mathcal{H}_{p}^{(i)}:=\mathrm{Red}^{-1}(\mathcal{T}_{p,\mathbb{Q}}^{(i-1)})$.
\end{remark}


\subsection{The \v{C}erednik-Drinfel'd Theorem}
Let $B$ be an indefinite quaternion algebra over $\mathbb{Q}$ of discriminant $Dp$, and let $\mathcal{O}_{B}=\mathcal{O}_{B}(N)\subseteq B$ be an Eichler order of level $N$. Once we have fixed a real matrix immersion $\Phi_{\infty}:B\hookrightarrow \mathrm{M}_{2}(\mathbb{R})$, we consider the following discrete subgroup of $\mathrm{PGL}_{2}(\mathbb{R})_{>0}\simeq\mathrm{PSL}_{2}(\mathbb{R})$:
\[
\Gamma_{\infty,+}:=\Phi_{\infty}(\{\alpha\in\mathcal{O}_{B}^{\times}\mid\,\mathrm{Nm}(\alpha)>0\})/\mathbb{Z}^{\times}.
\] 
By a fundamental result of Shimura \cite[Main Theorem I]{Shimura1967}, there exists a proper algebraic curve $X(Dp,N)$ over $\mathbb{Q}$ whose complex points are parametrised by a holomorphic bijective map 
$
J:\Gamma_{\infty,+}\backslash\mathcal{H}\stackrel{\sim}{\rightarrow} X(Dp,N)(\mathbb{C}),
$ 
and which is characterised, up to isomorphisms over $\mathbb{Q}$, by the arithmetic property that the values of $J$ at certain special parameters $\tau\in\Gamma_{\infty,+}\backslash\mathcal{H}$ are algebraic points $J(\tau)\in X(Dp,N)(\mathbb{Q}^{ab})$ defined over some ray class field of $\mathbb{Q}$. 
	The curve $X(Dp,N)$ satisfying this property is called \emph{the Shimura curve of discriminant $Dp$ and level $N$}.

Since $X(Dp,N)$ is a scheme of locally finite type over $\mathbb{Q}$ we can consider, on one hand, its complex analytification $(X(Dp,N)\otimes_{\mathbb{Q}}\mathbb{C})^{an}$ (given by Serre's GAGA functor), which is a complex manifold, and on the other hand its $p$-adic rigidification $(X(Dp,N)\otimes_{\mathbb{Q}}\mathbb{Q}_{p})^{rig}$ which is a rigid analytic variety over $\mathbb{Q}_{p}$ (cf. \cite{Bosch2014}). 
While the complex analytification is uniformised by the discrete cocompact subgroup $\Gamma_{\infty,+}\subseteq\mathrm{Aut}(\mathcal{H})$, as we have recalled above, the $p$-adic rigidification also admits a uniformisation by a discrete cocompact subgroup $\Gamma_{p}\subseteq\mathrm{PGL}_{2}(\mathbb{Q}_{p})\simeq\mathrm{Aut}(\mathcal{H}_{p})$, which is known as the $p$-adic uniformisation of the Shimura curve $X(Dp,N)$ (or as the \v{C}erednik-Drinfel'd uniformisation).
	The group $\Gamma_{p}$ is defined, following \v{C}erednik \cite[Theorem 2.1]{Cerednik1976B}, by \emph{interchanging the local invariants} $p$ and $\infty$ in the quaternion algebra $B$. 

	More specifically, let $H$ be the definite quaternion algebra of discriminant $D$, let $\mathcal{O}_{H}=\mathcal{O}_{H}(N)\subseteq H$ be an Eichler order of level $N$, and consider the localised order $\mathcal{O}_{H}[1/p]:=\mathcal{O}_{H}\otimes_{\mathbb{Z}}\mathbb{Z}[1/p]$ over $\mathbb{Z}[1/p]$. 
Once we have fixed a $p$-adic matrix immersion $\Phi_{p}:H\hookrightarrow\mathrm{M}_{2}(\mathbb{Q}_{p})$, let us consider the following discrete cocompact subgroups of $\mathrm{PGL}_{2}(\mathbb{Q}_{p})$:
\begin{enumerate}[(i)]
	\item $\Gamma_{p}:=\Phi_{p}(\mathcal{O}_{H}[1/p]^{\times})/\mathbb{Z}[1/p]^{\times},$
	\item $\Gamma_{p,+}:=\Phi_{p}(\{\alpha\in\mathcal{O}_{H}[1/p]^{\times}\mid\,v_{p}(\mathrm{Nm}_{H/\mathbb{Q}}(\alpha))\equiv 0\;\mathrm{mod}\,2\})/\mathbb{Z}[1/p]^{\times}.$
\end{enumerate}
	The \v{C}erednik-Drinfel'd Theorem can then be stated in the following way.

\begin{teo}\label{CD}
	There is an isomorphism of rigid analytic varieties over $\mathbb{Q}_{p}$
\[
\Gamma_{p}\backslash(\mathcal{H}_{p}\otimes_{\mathbb{Q}_{p}}\mathbb{Q}_{p^{2}})\simeq (X(Dp,N)\otimes_{\mathbb{Q}}\mathbb{Q}_{p})^{rig},
\]
and an isomorphism of rigid analytic varieties over $\mathbb{Q}_{p^{2}}$
\[
(\Gamma_{p,+}\backslash\mathcal{H}_{p})\otimes_{\mathbb{Q}_{p}}\mathbb{Q}_{p^{2}}\simeq 
(X(Dp,N)\otimes_{\mathbb{Q}}\mathbb{Q}_{p})^{rig}\otimes_{\mathbb{Q}_{p}}\mathbb{Q}_{p^{2}}.
\]
	In particular, the $p$-adic Shimura curve $X(Dp,N)\otimes_{\mathbb{Q}}\mathbb{Q}_{p}$ is a quadratic twist over $\mathbb{Q}_{p^{2}}$ of an algebraic curve $X_{p,+}$ such that $\Gamma_{p,+}\backslash\mathcal{H}_{p}\simeq X_{p,+}^{rig}.$
\qed
\end{teo}

\begin{remark}
	The first isomorphism of the theorem is defined over $\mathbb{Q}_{p}$. This means that $\mathcal{H}_{p}\otimes_{\mathbb{Q}_{p}}\mathbb{Q}_{p^{2}}$, which is naturally a rigid analytic variety over $\mathbb{Q}_{p^{2}}$, is considered as a rigid analytic variety over $\mathbb{Q}_{p}$. 
	In addition to this, it can be seen that the first isomorphism implies the second one (for more details cf.  \cite{JordanLivne1984}, \cite{deVera_thesis}, and \cite{PM_tesis}).
\end{remark}

\begin{remark}
Even if the $\mathbb{Z}$-order $\mathcal{O}_{H}$ is not always unique up to conjugation, the $\mathbb{Z}[1/p]$-order  $\mathcal{O}_{H}[1/p]$ is, since this satisfies Eichler's condition (cf. \cite[Corollaire 5.7]{Vigneras1980}). Thus the conjugacy class of $\Gamma_{p}$ inside $\mathrm{PGL}_{2}(\mathbb{Q}_{p})$ is well determined and gives a rigid analytic variety $\Gamma_{p}\backslash\mathcal{H}_{p}$ well-defined up to rigid analytic isomorphisms.
Similar arguments as in the complex case show that the isomorphism class over $\mathbb{Q}$ of $X(Dp,N)$ does not depend on the Eichler order chosen in the conjugacy class of $\mathcal{O}_{B}(N)$. 
\end{remark}

\begin{remark}
	The groups $\Gamma_{p}$ and $\Gamma_{p,+}$ are not always torsion-free but, as we will see, they admit torsion-free, normal and finite index subgroups (cf. Lemma \ref{teo_estructura_grupos}). Therefore the curve $X_{p,+}$ is actually a finite quotient of a Mumford curve (see Remark \ref{remark_finite_quotient}).
\end{remark}

	Now recall that since the curve $X(Dp,N)\otimes_{\mathbb{Q}}\mathbb{Q}_{p}$ is proper, by a theorem of Raynaud its rigidification can be interpreted as the generic fibre of a $\mathbb{Z}_{p}$-formal scheme, which is obtained as the completion of an adequate integral model of $X(Dp,N)$ along its special fibre.

\begin{defn}\label{Drinfeld_model}
	We call \emph{Drinfel'd integral model} over $\mathbb{Z}_{p}$ of the Shimura curve $X(Dp,N)$, the integral model $\mathcal{X}(Dp,N)$ over $\mathbb{Z}_{p}$ whose completion along its special fibre is the formal scheme
$\Gamma_{p}\backslash(\widehat{\mathcal{H}}_{p}\otimes_{\mathrm{Spf}\,\mathbb{Z}_{p}}\mathrm{Spf}\,\mathbb{Z}_{p^{2}})$ over $\mathrm{Spf}\,\mathbb{Z}_{p}$.
\end{defn}

Drinfel'd constructed this integral model in \cite{Drinfeld1976} as a solution of a moduli problem over $\mathbb{Z}_{p}$, extending the modular interpretation of $X(Dp,N)(\mathbb{C})$.  

\begin{cor}\label{coro_reduction}
The reduction-graph of the special fibre of $\mathcal{X}(Dp,N)$ is the graph $\Gamma_{p,+}\backslash\mathcal{T}_{p}$.
\end{cor}

\begin{proof}
Corollary \ref{coro_reduction} is a direct consequence of a combination of Theorems \ref{CD} and \ref{red_map}.
\end{proof}


\subsection{$p$-adic Schottky groups and good fundamental domains}
In this section we briefly present the theory of Mumford curves (cf. \cite{Mumford1972}) and relate it to the theory of $p$-adic uniformisation of Shimura curves. Moreover we recall the notion of \emph{good} fundamental domains, following \cite{Gerritzen_vanderPut1980}.



\begin{defn}\label{Def_Schottky_groups}
A subgroup $\Gamma\subseteq\mathrm{PGL}_{2}(\mathbb{Q}_{p})$ is called a \emph{$p$-adic Schottky group} if:
\begin{enumerate}[(i)]
\item
it is discontinuous, i.e. its associated set of limit points $\mathcal{L}_{\Gamma}$ is different from $\mathbb{P}^{1}(\mathbb{C}_{p})$,
\item
it is finitely generated,
\item
it has no elements of finite order different from the identity $\mathrm{I}_{2}$; equivalently, all transformations $\gamma\in\Gamma,\gamma\neq\mathrm{I}_{2}$, are hyperbolic.
\end{enumerate}
\end{defn}  

	Given a Schottky group $\Gamma\subseteq\mathrm{PGL}_{2}(\mathbb{Q}_{p})$, denote by $\widehat{\mathcal{H}}_{\Gamma}$ the admissible formal scheme over $\mathrm{Spf}\,\mathbb{Z}_{p}$ whose generic fibre is the rigid analytic variety $\mathcal{H}_{\Gamma}:=\mathbb{P}^{1,rig}\smallsetminus\mathcal{L}_{\Gamma}$, and by $\mathcal{T}_{\Gamma}$ the dual graph of its special fibre. Since $\mathcal{L}_{\Gamma}\subseteq\mathbb{P}^{1}(\mathbb{Q}_{p})$, we have that $\mathcal{H}_{\Gamma}$ is a subdomain of $\mathcal{H}$, and $\mathrm{Ver}(\mathcal{T}_{\Gamma})$ is a subset of $\mathrm{Ver}(\mathcal{T})$. We also know that, actually, $\mathcal{T}_{\Gamma}$ is a subtree of $\mathcal{T}$.
	In his celebrated paper \cite{Mumford1972} Mumford proved that for every Schottky group $\Gamma$ there exists a proper curve $\mathcal{C}_{\Gamma}$ over $\mathbb{Z}_{p}$ such that $\Gamma\backslash\widehat{\mathcal{H}}_{\Gamma}\simeq\widehat{\mathcal{C}}_{\Gamma}$ is an isomorphism of formal schemes over $\mathrm{Spf}\,\mathbb{Z}_{p}$. 
The curve $\mathcal{C}_{\Gamma}$ is a stable curve (in the sense of Deligne and Mumford) such that its special fibre $\mathcal{C}\otimes_{\mathbb{Z}_{p}}\mathbb{F}_{p}$ is $\mathbb{F}_{p}$-split degenerate and its dual graph is $\Gamma\backslash\mathcal{T}_{\Gamma}$ (\cite[Definition 3.2 and Theorem 3.3]{Mumford1972}). Mumford also proved that the correspondence $\Gamma\mapsto\mathcal{C}_{\Gamma}$ induces a bijection between the set of isomorphic classes of stable curves $\mathcal{C}$ over $\mathbb{Z}_{p}$ with $\mathbb{F}_{p}$-split degenerate special fibre, and the set of conjugacy classes of Schottky groups $\Gamma\subseteq\mathrm{PGL}_{2}(\mathbb{Q}_{p})$.  
Finally it can be proved that when $\Gamma\subseteq\mathrm{PGL}_{2}(K)$ is cocompact, then $\mathcal{L}_{\Gamma}=\mathbb{P}^{1}(\mathbb{Q}_{p})$, $\mathcal{H}_{\Gamma}=\mathcal{H}_{p}$, and $\mathcal{T}_{\Gamma}=\mathcal{T}_{p}$ (cf. \cite{PM_tesis}).

\begin{defn}\label{dom_fund_grupos_Schottky}
Let $\Gamma\subseteq\mathrm{PGL}_{2}(\mathbb{Q}_{p})$ be a Schottky group of rank $g$ and $S=\{\gamma_{1},\dots,\gamma_{g}\}$ a system of generators for $\Gamma$. 
A \emph{good fundamental domain} for $\Gamma$ with respect to $S$ is an admissible subdomain of $\mathcal{H}_{p}$,
$
\mathcal{F}_{\Gamma}=\mathbb{P}^{1}(\mathbb{C}_{p})\smallsetminus\cup_{i=1}^{2g}\mathbb{B}^{-}(\alpha_{i},\rho_{i}),
$
which satisfies the following conditions:
\begin{enumerate}[(i)]
\item
the centres $\alpha_{i}$ are in $\mathbb{P}^{1}(\mathbb{Q}_{p})$, for every $1\leq i\leq 2g$,
\item
the closed balls $\mathbb{B}^{+}_{i}(\alpha_{i},\rho_{i})$, for $1\leq i\leq 2g$, are pair-wise disjoint,
\item
for every $1\leq i\leq g$,
\[\begin{array}{cc}
\gamma_{i}(\mathbb{P}^{1}(\mathbb{C}_{p})\smallsetminus \mathbb{B}^{-}(\alpha_{i},\rho_{i}))=\mathbb{B}^{+}(\alpha_{i+g},\rho_{i+g}), & 
\gamma_{i}(\mathbb{P}^{1}(\mathbb{C}_{p})\smallsetminus \mathbb{B}^{+}(\alpha_{i},\rho_{i}))=\mathbb{B}^{-}(\alpha_{i+g},\rho_{i+g}).
\end{array}\]
\end{enumerate}
\end{defn}


	In \cite{Gerritzen1974} and \cite[Ch. I]{Gerritzen_vanderPut1980} the existence of good fundamental domains for Schottky groups is proved using the non-archimedean analogue of Ford's method of isometry circles.
	The fact that the domain $\mathcal{F}_{\Gamma}$ defined above is a fundamental domain (in the usual sense) for the action of $\Gamma$ is proved in \cite[Proposition 4.1]{Gerritzen_vanderPut1980}.

The following result is a $p$-adic analogue of a well-known result about discrete subgroups of $\mathrm{PSL}_{2}(\mathbb{R})$, first proved by Selberg (\cite[Lemma 8]{Selberg1960}). 
	Since this is at the base of the method that we present to find Mumford curves covering $p$-adic Shimura curves, we sketch a proof here.
 
\begin{teo}\label{teo_estructura_grupos}
Let $\Gamma\subseteq\mathrm{PGL}_{2}(\mathbb{Q}_{p})$ be a discontinuous and finitely generated group.  Then there exists a normal subgroup $\Gamma^{Sch}$ of finite index which is torsion-free. In particular $\Gamma^{Sch}$ is a $p$-adic Schottky group.
\end{teo}

\begin{proof}
Let $S$ be a subset of elements of finite order in $\Gamma$ such that every element of finite order in $\Gamma$ is conjugated to an element in $S$. It is easy to prove that the set $S$ can be taken to be finite (cf. \cite[Lemma 3.3.2]{Gerritzen_vanderPut1980}).
Since $\Gamma$ is finitely generated, we can find a finitely generated ring $R\subseteq\mathbb{Q}_{p}$ over $\mathbb{Z}$ such that $\Gamma\subseteq\mathrm{PGL}_{2}(R)\subseteq\mathrm{PGL}_{2}(\mathbb{Q}_{p})$.
Since $S$ is finite, there exists a maximal ideal $\mathfrak{m}\subseteq R$ such that, for every $\sigma\in S$, $\sigma\neq \mathrm{I}_{2}$, we have $\sigma-\mathrm{I}_{2}\notin\mathfrak{m}$. Therefore, the group 
\[
\Gamma^{Sch}:=\Gamma(\mathfrak{m}):=\{[\gamma]\in \Gamma\mid\,\gamma-\mathrm{I}_{2}\in\mathfrak{m}\}
\]
is a normal subgroup of $\Gamma$ of finite index (since $R/\mathfrak{m}$ is a finite field) and  from the fact that $\Gamma(\mathfrak{m})$ is normal it follows that $\Gamma(\mathfrak{m})$ does not contain any of the elements of finite order in $\Gamma$.
\end{proof}

\begin{remark}\label{remark_finite_quotient}
	This is one of the key steps in the proof of the \v{C}erednik-Drinfel'd Theorem, since it allows reducing to the case of Mumford curves through a finite cover.  
	By this result we find that the curve $X_{p,+}$ of Theorem \ref{CD} is a finite quotient of a Mumford curve. Therefore the $p$-adic Shimura curve is a quadratic twist of a finite quotient of a Mumford curve.
\end{remark}

\section{Modular arithmetic in definite quaternion algebras}\label{Sec2}

In this section we extend the notion of \textit{primary quaternions} introduced in \cite{Hurwitz1896} for the order of Hurwitz quaternions, to Eichler $\mathbb{Z}$-orders $\mathcal{O}$ over $\mathbb{Z}$ with $h(D,N)=1$. This notion will turn out to be crucial in order to prove a unique factorisation result (Theorem \ref{Zerlegungssatz}) in these orders, a result that extends the \textit{Zerlegunsgsatz} for the Hurwitz quaternions.

Let $H$ denote a definite quaternion algebra over $\mathbb{Q}$ of discriminant $D$ and let $\mathcal{O}\subset H$ denote an Eichler order over $\mathbb{Z}$ of level $N$. Fix a a place $\ell$ of $\mathbb{Q}$ and let $\mathcal{O}[1/\ell]:=\mathcal{O}\otimes_\mathbb{Z}\mathbb{Z}[1/\ell]$ denote the corresponding order over $\mathbb{Z}[1/\ell]$. In particular note that if $\ell=\infty$ then $\mathbb{Z}[1/\ell]=\mathbb{Z}$. 

Let $\mathfrak{a}\subseteq H$ be a $\mathbb{Z}[1/\ell]$-ideal (i.e. a $\mathbb{Z}[1/\ell]$-module of rank $4$). We denote by $\mathcal{O}_l(\mathfrak{a}):=\{\alpha\in H\mid\, \alpha\mathfrak{a}\subseteq\mathfrak{a}\}$ its left order and by $\mathcal{O}_r(\mathfrak{a}):=\{\alpha\in H\mid\, \mathfrak{a}\alpha\subseteq\mathfrak{a}\}$ its right order.
Then $\mathfrak{a}$ is said to be a \textit{left ideal} of $\mathcal{O}_l(\mathfrak{a})$ and a \textit{right ideal} of $\mathcal{O}_r(\mathfrak{a})$, or a \emph{one-sided} ideal of $\mathcal{O}_l(\mathfrak{a})$ and $\mathcal{O}_r(\mathfrak{a})$. The ideal $\mathfrak{a}$ is said to be \textit{integral} if it is contained both in $\mathcal{O}_l(\mathfrak{a})$ and $\mathcal{O}_r(\mathfrak{a})$, \emph{two-sided} if $\mathcal{O}_l(\mathfrak{a})=\mathcal{O}_r(\mathfrak{a})$ and \emph{principal} if there exists $\alpha\in\mathcal{O}$ such that $\mathfrak{a}=\mathcal{O}_{l}(\mathfrak{a}) \alpha=\alpha\mathcal{O}_r(\mathfrak{a})$. 

Two $\mathbb{Z}[1/\ell]$-ideals $\mathfrak{a},\mathfrak{b}\subseteq H$ are equivalent on the left (resp.\ on the right) if there exists $\alpha\in H^{\times}$ such that $\mathfrak{a}=\alpha\mathfrak{b}$ (resp. $\mathfrak{a}=\mathfrak{b}\alpha$).
It is immediate to see that if $\mathfrak{a}$ and $\mathfrak{b}$ are equivalent on the left, then they have the same right order $\mathcal{O}[1/\ell]$, and the set of classes of right ideals of $\mathcal{O}[1/\ell]$ is called the \emph{set of right ideal classes} of the order $\mathcal{O}[1/\ell]$. 
Analogously we define the \emph{set of left ideal classes} of $\mathcal{O}[1/\ell]$.
The set of left and right ideal classes are in bijection,
and its cardinality is the \emph{one-sided ideal class number} of $\mathcal{O}[1/\ell]$. Moreover this class number does not depend on the conjugacy class of $\mathcal{O}[1/\ell]$ inside $H$, and so it is denoted by $h(D,N)$. 
Note that the number $h(D,N)$ also depends on the base ring $\mathbb{Z}[1/\ell]$.
Actually, when $\mathcal{O}[1/\ell]$ satisfies Eichler's condition (i.e., when the algebra $H$ is not ramified at $\ell$) then, by the strong approximation theorem, this number coincides with the strict ideal class number of $\mathbb{Q}$ (cf. \cite{Eichler1938_crelle}).

Let $\mathfrak{a}\subseteq\mathcal{O}$ be an integral right-sided ideal (resp. left-sided) of $\mathcal{O}$, and let $(a):=\mathfrak{a}\cap\mathbb{Z}$. Then $\mathcal{O}/\mathfrak{a}$ is a finitely generated right-module (resp. left-module) over $\mathbb{Z}/a\mathbb{Z}$.
For $\alpha,\beta\in\mathcal{O}$, the quaternion $\alpha$ is said to be \emph{congruent to $\beta$ modulo $\mathfrak{a}$}, and one writes $\alpha\equiv\beta\ \mathrm{mod}\ \mathfrak{a}$, if $\alpha-\beta\in\mathfrak{a}$.
 In particular, if $\mathfrak{a}$ is a two-sided integral ideal of $\mathcal{O}$, then $\mathcal{O}/\mathfrak{a}$ is a (possibly) non-commutative ring, which is called the \emph{ring of quaternion classes modulo $\mathfrak{a}$}.

\begin{notation}
Let $\mathfrak{a}=\gamma\mathcal{O}$ be an integral ideal and consider the natural projection
$
\lambda_r:\mathcal{O}/\mathrm{Nm}(\gamma)\mathcal{O}\twoheadrightarrow \mathcal{O}/\gamma\mathcal{O}.
$
We denote by $(\mathcal{O}/\gamma\mathcal{O})_r^{\times}$ the image of $(\mathcal{O}/\mathrm{Nm}(\gamma)\mathcal{O})^\times$ under $\lambda_r$.
Analogously, if $\mathfrak{a}=\mathcal{O}\gamma$ we denote by $\lambda_\ell$ the associated projection and $(\mathcal{O}/\mathcal{O}\gamma)_\ell^{\times}:=\mathrm{Im}(\lambda_\ell)$.
\end{notation}

\begin{remark}\label{xi=n}
Note that when $\gamma=m\in\mathbb{Z}$, then $m\mathcal{O}=\mathcal{O}m$, and the set $(\mathcal{O}/m\mathcal{O})_r^{\times}=(\mathcal{O}/\mathcal{O}m)_\ell^{\times}$ coincides with the group of units $(\mathcal{O}/m\mathcal{O})^{\times}$ of the quotient ring $\mathcal{O}/m\mathcal{O}$, since the projection of rings $\lambda:\mathcal{O}/m^2\mathcal{O}\rightarrow \mathcal{O}/m\mathcal{O}$ restricts to a projection on the associated unit groups.

Let $\{1,\theta_{1},\theta_{2},\theta_{3}\}$ be a $\mathbb{Z}$-basis of $\mathcal{O}$, then
$\mathcal{O}/m\mathcal{O}=\{[\alpha]=a_{0}+a_{1}\theta_{1}+a_{2}\theta_{2}+a_{3}\theta_{3}\mid\, a_{i}\in\mathbb{Z}/m\mathbb{Z}\}$
is a free $\mathbb{Z}/m\mathbb{Z}$-module of rank 4. Its group of units is
$(\mathcal{O}/m\mathcal{O})^\times=
\{[\alpha]\in\mathcal{O}/m\mathcal{O}\mid(\mathrm{Nm}(\alpha),m)=1\}.$
\end{remark}

Let $\xi\in\mathcal{O}$. The group $\mathcal{O}^\times$ acts on $(\mathcal{O}/\xi\mathcal{O})_r^{\times}$:
take $\alpha+\xi\mathcal{O}\in(\mathcal{O}/\xi\mathcal{O})_r^\times$ and $u\in\mathcal{O}^\times$. Then we define
$(\alpha+\xi\mathcal{O})u:=\alpha u+\xi\mathcal{O}=\lambda_r(\alpha u+\mathrm{Nm}(\xi)\mathcal{O})\in(\mathcal{O}/\xi\mathcal{O})_r^{\times}.$
Two elements $\alpha+\xi\mathcal{O}$ and $\beta+\xi\mathcal{O}$ are equivalent if there exists $u\in\mathcal{O}^\times$ such that $(\alpha+\xi\mathcal{O})u=\beta+\xi\mathcal{O}$. We denote by $[\alpha]$ the class of $\alpha+\xi\mathcal{O}$.

The following definition generalises a very important notion in \cite{Hurwitz1896}. In what follows, we take $\mathcal{O}=\mathcal{O}_H(N)\subseteq H$ to be an Eichler order over $\mathbb{Z}$ of level $N$ with $h(D,N)=1$.

\begin{defn}
	Take a quaternion $\xi\in\mathcal{O}$, and let $\mathcal{P}=\{[\alpha_1],\ldots,[\alpha_r]\}\subseteq(\mathcal{O}/\xi\mathcal{O})_r^\times$ be a system of representatives of $(\mathcal{O}/\xi\mathcal{O})_r^\times$ for the right-multiplication action given by $\mathcal{O}^\times$. Suppose that:
	\begin{enumerate}[$(i)$]
		\item $\mathcal{O}^\times$ acts freely on $(\mathcal{O}/\xi\mathcal{O})_r^\times$ if $2\notin\xi\mathcal{O}$,
		\item $\mathcal{O}^\times/\mathbb{Z}^\times$ acts freely on $(\mathcal{O}/\xi\mathcal{O})_r^\times$ if $2\in\xi\mathcal{O}$.
	\end{enumerate}
	Then we call $\mathcal{P}$ a \textit{$\xi$-primary class set} for $\mathcal{O}$. A quaternion $\alpha$ that belongs to some class in $\mathcal{P}$ is called a \emph{$\xi$-primary} quaternion with respect to $\mathcal{P}$. 
\end{defn}

\begin{remark}\label{obs1}
	Assume that $\mathcal{O}$ admits a $\xi$-primary class set $\mathcal{P}=\{[\alpha_1],\ldots,[\alpha_r]\}$ and take $\alpha\in\mathcal{O}$ with $[\alpha]=\alpha+\xi\mathcal{O}\in(\mathcal{O}/\xi\mathcal{O})_r^\times$. Then there exists a unique $1\leq i\leq r$ and a unit $u\in\mathcal{O}^\times$ such that $[\alpha]u=[\alpha_i]\in\mathcal{P}$. If $2\notin\xi\mathcal{O}$, then $u$ is unique, otherwise $u$ is unique up to sign.
	In particular, if $\mathcal{P}=\{[1]\}\subseteq\mathcal{O}/\xi\mathcal{O}$, then $\alpha$ is $\xi$-primary with respect to $\{[1]\}$ if and only if $\alpha\equiv1\ \mathrm{mod}\ \xi\mathcal{O}$.
	
\end{remark}

The notion of $\xi$-primary quaternions turns out to be very important to have a unique factorisation in definite Eichler orders.
In \cite{Hurwitz1896} the author takes $\xi=2(1+i)$ and shows that the order of Hurwitz quaternions $\mathcal{O}$ admits the $\xi$-primary class set $\{[1],[1+2\rho]\}\subseteq(\mathcal{O}/2(1+i)\mathcal{O})^\times$. Using this, he proves a unique factorisation result or \emph{Zerlegungssatz} in this Eichler order. Our next goal is to extend this \emph{Zerlegungssatz} to all definite Eichler orders with $h(D,N)=1$. Moreover in the next section we will find a $\xi$-primary class set for most of these orders.

\begin{defn}\label{units_property}
	Take $\xi\in\mathcal{O}$. We say that $\xi$ satisfies the \textit{right-unit property} in $\mathcal{O}$ if:	
	\begin{enumerate}[$(i)$]
	\item the map $\varphi:\mathcal{O}^{\times}\rightarrow (\mathcal{O}/\xi\mathcal{O})_r^{\times},\  u \mapsto \lambda_r(u+\mathrm{Nm}(\xi)\mathcal{O})$,
	is a bijection if $2\notin\xi\mathcal{O}$.
	
	\item the map $\varphi:\mathcal{O}^{\times}/\mathbb{Z}^{\times}\rightarrow (\mathcal{O}/\xi\mathcal{O})_r^{\times},\  u \mapsto \lambda_r(u+\mathrm{Nm}(\xi)\mathcal{O})$,
	is a bijection if $2\in\xi\mathcal{O}$.
	\end{enumerate}
\end{defn}

\begin{remark}\label{xi2}
	In the second case we have that either $\xi\mathcal{O}=2\mathcal{O}$, or $\mathrm{Nm}(\xi)=2$ or $\mathrm{Nm}(\xi)=1$, so in particular the surjection $\lambda_r$ is either $\mathcal{O}/4\mathcal{O}\twoheadrightarrow \mathcal{O}/2\mathcal{O}$ or $\mathcal{O}/2\mathcal{O}\twoheadrightarrow\mathcal{O}/\xi\mathcal{O}$ or $\mathcal{O}\twoheadrightarrow \{1\}$.
\end{remark}

\begin{remark}\label{group_structure}
If $\xi\in\mathcal{O}$ satisfies the right-unit property in $\mathcal{O}$, then the bijection of the definition induces a multiplicative group structure on the set $(\mathcal{O}/\xi\mathcal{O})_r^{\times}=\mathrm{Im}(\lambda_r)$.
\end{remark}

When $\xi\mathcal{O}$ is a two-sided integral principal ideal, then $\lambda_r=\lambda_\ell$ is an epimorphism of rings and $(\mathcal{O}/\xi\mathcal{O})^\times$ is the unit group of the ring $\mathcal{O}/\xi\mathcal{O}$.

\begin{lemma}\label{units_mod}
	Let $\xi\mathcal{O}$ be an integral two-sided ideal such that $\varphi:\mathcal{O}^{\times}\rightarrow(\mathcal{O}/\xi\mathcal{O})^{\times}$ $\emph{(}$resp. $\mathcal{O}^\times/\mathbb{Z}^\times\rightarrow(\mathcal{O}/\xi\mathcal{O})^\times)$, $\varphi(u)=[u]$ is a well defined group isomorphism if $2\notin\xi\mathcal{O}$ $\emph{(}$resp. if $2\in\xi\mathcal{O})$.
	Then $(\mathcal{O}/\xi\mathcal{O})_r^\times=(\mathcal{O}/\xi\mathcal{O})^\times$ and $\xi$ satisfies the right-unit property in $\mathcal{O}$.
\end{lemma}

\begin{proof}
	If $\xi=m\in\mathbb{Z}$ the result is clear (cf. Remark \ref{xi=n}).
	Suppose that $\xi\notin\mathbb{Z}$ with $\mathrm{Nm}(\xi)=m$, and consider the natural projection $\lambda_r:\mathcal{O}/m\mathcal{O}\twoheadrightarrow\mathcal{O}/\xi\mathcal{O},\ \alpha+m\mathcal{O} \mapsto \alpha+\xi\mathcal{O}$ which restricts to the unit group, so $\mathrm{Im}(\lambda_r)=(\mathcal{O}/\xi\mathcal{O})_r^\times\subseteq(\mathcal{O}/\xi\mathcal{O})^\times$. We want this restriction to be still surjective. 
By assumption, if $2\notin\xi\mathcal{O}$ (resp. if $2\in\xi\mathcal{O}$) we have the group isomorphism
	$$\varphi:\mathcal{O}^{\times}\rightarrow(\mathcal{O}/\xi\mathcal{O})^\times\quad(\mbox{resp. }/\mathcal{O}^\times\mathbb{Z}^{\times}\rightarrow(\mathcal{O}/\xi\mathcal{O})^\times)\quad\varphi(u)=[u].$$
	Take $[\alpha]\in(\mathcal{O}/\xi\mathcal{O})^\times$. Then there exists $u\in\mathcal{O}^\times$ (resp. $\mathcal{O}^\times/\mathbb{Z}^\times$) such that $\varphi(u)=[u]=[\alpha]$. Moreover, since $\mathrm{Nm}(u)=1$, we have $u+m\mathcal{O}\in(\mathcal{O}/m\mathcal{O})^\times$ (cf. Remark \ref{xi=n}), and $\lambda_r(u+m\mathcal{O}):=u+\xi\mathcal{O}=\alpha+\xi\mathcal{O}=[\alpha]$, so $(\mathcal{O}/\xi\mathcal{O})_r^\times=(\mathcal{O}/\xi\mathcal{O})^\times$. Thus $\xi$-satisfies the right-unit property.
\end{proof}

\begin{lemma}\label{isom}
	Take $\xi\in\mathcal{O}$.
	Then $\mathcal{P}=\{[1]\}$ is a $\xi$-primary class set for $\mathcal{O}$ if and only if $\xi$ satisfies the right-unit property.
\end{lemma}

\begin{proof}
	$\#\mathcal{P}=\#(\mathcal{O}/\xi\mathcal{O})_r^\times/\#\mathcal{O}^\times$ if $2\notin\xi\mathcal{O}$ and $\#\mathcal{P}=\#(\mathcal{O}/\xi\mathcal{O})_r^\times/\#(\mathcal{O}^\times/\mathbb{Z}^\times)$ otherwise.
\end{proof}

The right-unit property turns out to be a very powerful condition. We will be particularly interested in integral ideals $\xi\mathcal{O}$ such that $2\in\xi\mathcal{O}$, since these are the useful ones in section 3.
In this case we will find quaternions with the right-unit property in all definite Eichler orders over $\mathbb{Z}$ with $h(D,N)=1$ (except for two cases, for which no such $\xi$ exists).

In Table \ref{Table_xi} we list all definite Eichler orders $\mathcal{O}\subseteq H$ of level $N$ (up to conjugation) with $h(D,N)=1$ that have some $\xi\in\mathcal{O}$ satisfying the right-unit property with $2\in\xi\mathcal{O}$. This list coincides with that of all definite Eichler orders over $\mathbb{Z}$ with $h(D,N)=1$, except for the values $(D,N)=(2,5)$ and $(7,1)$ (cf. \cite[Table 8.2]{KirschmerVoight2010}).
To do this, we fix a basis for $H$ containing each order, an integral basis for each order, and then express $\xi$ on this basis. Although the element $\xi$ is not unique, the existence of $\xi$ does not depend on the integral basis chosen for $\mathcal{O}$.

\begin{table}[h!]
\captionsetup{font=footnotesize}
\centering
\vspace{-0.5em}
\scalebox{0.68}{
\begin{tabular}{|c|c|c|c|c|c|c|} \hline
	\multicolumn{1}{|c|}{$D$} & $H$ & $N$ & $\mathcal{O}$ & \multicolumn{1}{|c|}{$\#(\mathcal{O}^{\times}/\mathbb{Z}^{\times})$} & \multicolumn{1}{|c|}{$\xi$} & \multicolumn{1}{|c|}{$\mathrm{Nm}(\xi)$}\\ \hline \hline
	2 & $\left(\frac{-1,-1}{\mathbb{Q}}\right)$ & 1 & $\mathbb{Z}\left[1, i, j, \frac{1}{2}(1+i+j+k)\right]$ & 12 & $2$ & 4 \\ 
	& $\vphantom{\left(\frac{-1,-1}{\mathbb{Q}}\right)}$ & 3 & $\mathbb{Z}\left[ 1, 3i, -2i + j, \frac{1}{2}(1-i+j+k)\right]$ & 3 & $(-i+k)$ & 2 \\ 
	& $\vphantom{\left(\frac{-1,-1}{\mathbb{Q}}\right)}$ & 9 & $\mathbb{Z}\left[ 1, 9i, -4i + j, \frac{1}{2}(1-3i+j+k) \right]$ & 1 & $1$ & 1 \\ 
	& $\vphantom{\left(\frac{-1,-1}{\mathbb{Q}}\right)}$ & 11 & $\mathbb{Z}\left[ 1, 11i, -10i + j, \frac{1}{2}(1-3i+j+k)\right]$ & 1 & $1$ & 1 \\ \hline\hline
	3 & $\left(\frac{-1,-3}{\mathbb{Q}}\right)$ & 1 & $\mathbb{Z}\left[ 1, i, \frac{1}{2}(i+j), \frac{1}{2}(1+k) \right]$ & 6 & $2$ & 4 \\ 
	& $\vphantom{\left(\frac{-1,-1}{\mathbb{Q}}\right)}$ & 2 & $\mathbb{Z}\left[ 1, 2i, \frac{1}{2}(-i + j), \frac{1}{2}-i+\frac{1}{2}k\right]$ & 2 & $\frac{1}{2}(-1-i-j+k)$ & 2 \\ 
	& $\vphantom{\left(\frac{-1,-1}{\mathbb{Q}}\right)}$ & 4 & $\mathbb{Z}\left[ 1, 4i, \frac{1}{2}(-5i+j), \frac{1}{2} - 3i + \frac{1}{2}k \right]$ & 1 & 1 & 1 \\  \hline\hline
	5 & $\left(\frac{-2,-5}{\mathbb{Q}}\right)$& 1 & $\mathbb{Z}\left[ 1, \frac{1}{2}(1+i+j), j, \frac{1}{4}(2+i+k)\right]$ & 3 & $\frac{1}{2}(-1+i-j)$ & 2 \\
	& $\vphantom{\left(\frac{-1,-1}{\mathbb{Q}}\right)}$ & 2 & $\mathbb{Z}\left[ 1, 1 + i + j, \frac{1}{2}(-1-i+j), \frac{1}{4}(-i-2j+k) \right]$ & 1& 1 & 1 \\ \hline\hline
	13 & $\left(\frac{-2,-13}{\mathbb{Q}}\right)$ & 1 & $\mathbb{Z}\left[ 1, \frac{1}{2}(1+i+j), j, \frac{1}{4}(2+i+k) \right]$ & 1 & 1 & 1 \\  \hline
\end{tabular}
}
\caption{Definite Eichler orders $\mathcal{O}$ with an element $\xi\in\mathcal{O}$ satisfying the right-unit property.}\label{Table_xi}
\vspace{-1.5em}
\end{table}

\begin{defn}
	Let $\{1,\theta_1,\theta_2,\theta_3\}$ be an integral basis of $\mathcal{O}$. A quaternion $\alpha=a_0+a_1\theta_1+a_2\theta_2+a_3\theta_3\in\mathcal{O}$ is \emph{primitive} if the ideal generated by its coordinates, $(a_0,a_1,a_2,a_3)\subseteq\mathbb{Z}$, is the entire ring $\mathbb{Z}$. Note that this definition does not depend on the integral basis chosen.
	A nonzero quaternion $\pi\in\mathcal{O}\smallsetminus\mathcal{O}^\times$ is \emph{irreducible} if, whenever $\pi=\alpha\beta$, then either $\alpha\in\mathcal{O}^\times$ or $\beta\in\mathcal{O}^\times$.
\end{defn}


\begin{lemma}\label{primes}
	Let $\mathcal{O}\subseteq H$ be an  Eichler order of level $N$ with $h(D,N)=1$. A primitive quaternion $\pi\in\mathcal{O}$ is irreducible if and only if its norm is a prime integer.
\end{lemma}

\begin{proof}
If $\mathrm{Nm}(\pi)$ is prime then it is obvious that $\pi$ has to be irreducible. Reciprocally, let us assume the $\pi\in\mathcal{O}$ is irreducible and that $p\mid\mathrm{Nm}(\pi)$ is a prime factor of the norm. We are going to show that, when $\pi$ is primitive, we have strict inclusions
$
p\mathcal{O}\subsetneq p\mathcal{O}+\pi\mathcal{O}\subsetneq \mathcal{O}.
$
If the first inclusion were an equality, then $p\mid \pi$ and, since $\pi$ is irreducible, $p\in\mathcal{O}^\times$ or $\pi\mathcal{O}=p\mathcal{O}$. Both options lead to a contradiction, since $p$ is prime (of norm $\mathrm{Nm}(p)=p^{2}$) and $\pi$ is assumed to be primitive. If the second inclusion were an equality then it is an easy computation to see that $p\mid\mathrm{Nm}(x)$, for every $x\in p\mathcal{O}+\pi\mathcal{O}=\mathcal{O}$, a contradiction. Since the order $\mathcal{O}$ has one-sided ideal class number equal to 1, then $p\mathcal{O}+\pi\mathcal{O}=\alpha\mathcal{O}$, for some $\alpha\in\mathcal{O}\smallsetminus\mathcal{O}^\times$ and then $\pi=\alpha\beta$ for some $\beta\in\mathcal{O}$. Moreover $\beta\in\mathcal{O}^\times$ since $\pi$ is irreducible. Hence $p\in p\mathcal{O}+\pi\mathcal{O}=\alpha\mathcal{O}=\pi\mathcal{O}$, i.e., $\pi\mid p$ and $\mathrm{Nm}(\pi)\mid \mathrm{Nm}(p)=p^{2}$. Since $\pi$ is irreducible its norm $\mathrm{Nm}(\pi)$ cannot be equal to $1$ and, since $\pi$ is primitive, it cannot be $\mathrm{Nm}(\pi)=p^{2}$ either. Actually $\mathrm{Nm}(\pi)=p^{2}$ implies $p=\pi\varepsilon$ for a unit $\varepsilon\in\mathcal{O}^\times$ and so $\pi=p\varepsilon^{-1}$. Finally the only possible option is $\mathrm{Nm}(\pi)=p$.
\end{proof}

\begin{lemma}\label{auxiliar}
	Let $\mathcal{O}$ be an Eichler order with $h(D,N)=1$ and take $\alpha,\beta,\gamma\in\mathcal{O}$, where $\mathrm{Nm}(\beta)=p$ is prime. If $p\mid\alpha\beta\gamma$ and $p\nmid\alpha\beta$, then $p\mid\beta\gamma$.
\end{lemma}

\begin{proof}
	Consider the ideal $\overline{\beta}\mathcal{O}+\gamma\mathcal{O}=\delta\mathcal{O}$, which is principal by assumption. Then $\mathrm{Nm}(\delta)\mid\mathrm{Nm}(\overline{\beta})=p$, which gives $\mathrm{Nm}(\delta)=1$ or $p$. If $\mathrm{Nm}(\delta)=1$ then $\delta\in\mathcal{O}^\times$, so there exist $x,y\in\mathcal{O}$ such that $1=\overline{\beta}x+\gamma y$. This implies that $\alpha\beta=\alpha\beta(\overline{\beta}x+\gamma y)=\alpha\beta\overline{\beta}x+\alpha\beta\gamma y=\alpha px+\alpha\beta\gamma y$, so $p\mid\alpha\beta$, which is a contradiction. Thus $\delta=\overline{\beta}\varepsilon$, for some $\varepsilon\in\mathcal{O}^\times$. Since $\delta\mathcal{O}=\overline{\beta}\mathcal{O}+\gamma\mathcal{O}$, this gives that $\gamma=\overline{\beta}z$, for some $z\in\mathcal{O}$. Finally, we obtain that $p\mid\beta\overline{\beta}z=\beta\gamma$.
\end{proof}

\begin{teo}[Zerlegungssatz]\label{Zerlegungssatz}
Let $H$ be a definite quaternion algebra of discriminant $D$ and let $\mathcal{O}$ be an Eichler order over $\mathbb{Z}$ of level $N$ with $h(D,N)=1$. Let $\xi\in\mathcal{O}$ be an integral quaternion such that $\mathcal{O}/\xi\mathcal{O}$ contains a $\xi$-primary class set $\mathcal{P}$. 
Take $\alpha\in\mathcal{O}$ a primitive and $\xi$-primary quaternion with respect to $\mathcal{P}$ such that its norm has a decomposition in prime factors
$
\mathrm{Nm}(\alpha)=p_{1}\cdot\,\dots\,\cdot p_{s}.
$
Then $\alpha$ admits a decomposition in primitive irreducible and $\xi$-primary quaternions with respect to $\mathcal{P}$:
$
\alpha=\pi_{1}\cdot\,\dots\,\cdot \pi_{s}
$
with $\mathrm{Nm}(\pi_{i})=p_{i}$, for every $1\leq i \leq s$. Moreover, if $2\notin\xi\mathcal{O}$ this decomposition is unique, and if $2\in\xi\mathcal{O}$, the decomposition is unique up to sign.
\end{teo}

\begin{proof}
First consider the integral right ideal $p_{1}\mathcal{O}+\alpha\mathcal{O}\subseteq\mathcal{O}$ of $\mathcal{O}$, which is principal by assumption, so $p_{1}\mathcal{O}+\alpha\mathcal{O}=\pi_{1}\mathcal{O}$, for some $\pi_{1}\in \mathcal{O}$ uniquely determined up to right multiplication by a unit. Let us see that $\pi_1$ is then irreducible and $\mathrm{Nm}(\pi_1)=p_1$. Since $p_1\mathcal{O}\subseteq\pi_1\mathcal{O}$, we have that $\pi_1\mid p_1$, and then $\mathrm{Nm}(\pi_{1})\mid \mathrm{Nm}(p_{1})=p_{1}^{2}$. If $\mathrm{Nm}(\pi_{1})=1$, then $p_{1}\mathcal{O}+\alpha\mathcal{O}=\mathcal{O}$, and since every $x\in p_{1}\mathcal{O}+\alpha\mathcal{O}$ has norm such that $p_1\mid\mathrm{Nm}(x)$ one obtains a contradiction. If $\mathrm{Nm}(\pi_{1})=p_{1}^{2}=\mathrm{Nm}(p_1)$ then, since we also have $\pi_1\mid p_1$, we conclude that $\pi_1=p_1\varepsilon$, for some $\varepsilon\in\mathcal{O}^\times$. But then $p_{1}\mathcal{O}+\alpha\mathcal{O}=p_{1}\mathcal{O}$, so $p_{1}\mid \alpha$, which contradicts the primitivity of $\alpha$. Thus we have shown that $\mathrm{Nm}(\pi_{1})=p_{1}$, so it is primitive, and then we know by Lemma \ref{primes} that $\pi_{1}$ is irreducible.
Moreover, since by assumption we have a $\xi$-primary class set $\mathcal{P}\subseteq\mathcal{O}/\xi\mathcal{O}$, there exists a unit $\varepsilon_1\in\mathcal{O}^\times$ such that $\pi_1\varepsilon_1$ is $\xi$-primary, and this unit is unique if $2\notin\xi\mathcal{O}$ and is unique up to sign if $2\in\xi\mathcal{O}$.
Thus we obtain a decomposition $\alpha=\pi_1 \varepsilon_1\varepsilon_1^{-1}\alpha_1$, for some $\alpha_{1}\in\mathcal{O}$, where $\pi_1\varepsilon_1$ is primitive, $\xi$-primary and irreducible. Moreover
$\varepsilon_1^{-1}\alpha_1$ is also a primitive quaternion.

Next we apply the same argument to $\varepsilon_1^{-1}\alpha_{1}\in\mathcal{O}$, whose norm is $\mathrm{Nm}(\varepsilon_1^{-1}\alpha_{1})=p_{2}\dots p_{s}$, and we obtain $\varepsilon_1^{-1}\alpha_1=\pi_2\varepsilon_2\varepsilon_2^{-1}\alpha_2$ with $\pi_2\varepsilon_2$ primitive, $\xi$-primary and irreducible. Iterating the process we find primitive and irreducible quaternions $\pi_{1}\varepsilon_1,\dots,\pi_s\varepsilon_s$ in $\mathcal{O}$ with $\mathrm{Nm}(\pi_{i})=p_{i}$, such that $\alpha=\pi_1\varepsilon_1\ldots\pi_s\varepsilon_s$, and these determine a factorisation 
$\alpha=\pi_{1}\varepsilon_1\cdot\ldots\cdot\pi_{s}\varepsilon_s,$
where $\pi_1\varepsilon_1,\ldots,\pi_s\varepsilon_s$ are primitive and $\xi$-primary quaternions with $\mathrm{Nm}(\pi_i)=p_i$, and every $\varepsilon_i\in\mathcal{O}^{\times}$ in the decomposition is uniquely determined by each $\pi_i$ (up to sign if $2\in\xi\mathcal{O}$).

Finally suppose that we have two factorisations in primitive $\xi$-primary irreducible elements:
$$\alpha=\pi_1\cdot\ldots\cdot\pi_s=\sigma_1\cdot\ldots\cdot\sigma_s,\quad\mbox{for }s\geq2,$$
with $\mathrm{Nm}(\pi_i)=\mathrm{Nm}(\sigma_i)=p_i$, for $1\leq i\leq s$. Then $\pi_1\ldots\pi_s\overline{\pi}_s=\sigma_1\ldots\sigma_s\overline{\pi}_s$, so $p_s\mid\sigma_1\ldots\sigma_s\overline{\pi}_s=(\sigma_1\ldots\sigma_{s-1})\sigma_s\overline{\pi}_s$. Since $\alpha$ is primitive, $p_s\nmid(\sigma_1\ldots\sigma_{s-1})\sigma_s$. Thus, by Lemma \ref{auxiliar}, we have that $p_s\mid\sigma_s\overline{\pi}_s$. Now $\sigma_s\overline{\pi}_s=\varepsilon p_s=\varepsilon\pi_s\overline{\pi}_s$, for some $\varepsilon\in\mathcal{O}$, so $\sigma_s=\varepsilon\pi_s$ and $\varepsilon$ is a unit. Since both $\pi_s$ and $\sigma_s$ are primitive and $\xi$-primary, we obtain that $\varepsilon=1$ if $2\notin\xi\mathcal{O}$, and $\varepsilon=\pm1$ otherwise.
\end{proof}

\section{A method for finding Mumford curves covering Shimura curves}\label{Sec3}

In this section we state and prove the main result of this article (Theorem \ref{Theorem_Schottky_group}). This gives a specifit way to find generators for certain Schottky groups which are subgroups of the $p$-adic quaternion groups uniformising $p$-adic Shimura curves. In particular, we will be able to construct good fundamental domains for the associated Mumford curves and, consequently, to obtain the stable reduction-graphs of these curves (Theorem \ref{Theorem_Mumford_curve}). Finally, we will see that these  Mumford curves are arithmetic coverings of certain families of $p$-adic Shimura curves (Corollary \ref{Corollary}).

\subsection{Systems of generators for arithmetic Schottky groups}

Let $H$ be a definite quaternion algebra over $\mathbb{Q}$ of discriminant $D$ and $p\nmid D$ an odd prime. Let $\mathcal{O}=\mathcal{O}_H(N)\subseteq H$ be an Eichler order of level $N$ coprime to $p$. Denote by $\mathcal{O}[1/p]:=\mathcal{O}\otimes_{\mathbb{Z}}\mathbb{Z}[1/p]$ the corresponding Eichler order over $\mathbb{Z}[1/p]$, which is unique up to conjugation, and by $\mathcal{O}[1/p]^{\times}$ its group of units, which is the group of quaternions in $\mathcal{O}[1/p]$ with norm equal to some integral power of $p$. Finally denote by $\mathcal{O}[1/p]^{\times}_{+}$ the corresponding group of \emph{positive} units, i.e.,
\[
\mathcal{O}[1/p]^{\times}_{+}:=\{\alpha\in\mathcal{O}[1/p]^{\times}\mid v_{p}(\mathrm{Nm}(\alpha))\equiv 0\;\mathrm{mod}\;2\}.
\]
This is an index 2 subgroup of $\mathcal{O}[1/p]^{\times}$. A system of representatives for $\mathcal{O}[1/p]^{\times}_{+}/\mathbb{Z}[1/p]^{\times}$ is given by the group of elements $\alpha\in\mathcal{O}[1/p]^{\times}$ with norm $1$, in analogy to the arquimedean case.

In the following result we apply the \emph{Zerlegungssatz}, valid for Eichler orders $\mathcal{O}$ with $h(D,N)=1$ and with a $\xi$-primary class set, in order to give a unique factorisation in the group $\mathcal{O}[1/p]^\times$.

\begin{teo}\label{Dec}
Assume that $\mathcal{O}$ has $h(D,N)=1$ and let $\xi\in\mathcal{O}$ be an integral quaternion such that $\mathcal{O}/\xi\mathcal{O}$ contains a $\xi$-primary class set $\mathcal{P}$. Then every $\alpha\in\mathcal{O}[1/p]^{\times}$ can be decomposed as
$$
\alpha=p^{n}\cdot\prod_{i=1}^{r}\beta_{i}\cdot\varepsilon, 
$$
for a unique $\varepsilon\in\mathcal{O}^{\times}$ $($up to sign if $2\in\xi\mathcal{O})$, unique $n\in\mathbb{Z}$ and for unique $\beta_{1},\dots,\beta_{r}\in\mathcal{O}$ primitive $\xi$-primary quaternions with respect to $\mathcal{P}$ with $\mathrm{Nm}(\beta_{i})=p$, for every $1\leq i\leq r$, and such that no factor of the form $\beta_i\cdot\beta_{i+1} = p$ appears in the product.
\end{teo}
		
\begin{proof}
Let us take $\alpha\in\mathcal{O}[1/p]^\times$. Since the norm maps $\mathcal{O}[1/p]^{\times}$ into $\mathbb{Z}[1/p]^{\times}$ and the order $\mathcal{O}$ is definite, then $\mathrm{Nm}(\alpha)=p^{s}$, for some $s\in\mathbb{Z}$. Let $l\geq 0$ be the smallest non-negative integer such that $p^l\alpha\in\mathcal{O}$. Let $m\geq 1$ be the greatest common divisor of the integral coordinates of $p^l\alpha$ in some integral basis of $\mathcal{O}$, and put 
$\beta:=\frac{p^l\alpha}{m}$. Then $\mathrm{Nm}(\beta)=\frac{p^{2l+s}}{m^2}\in\mathbb{Z}$, so $m$ is a power of $p$, say $m=p^{t}$, and $\alpha= p^{t-l}\beta$.
Now $\beta\in\mathcal{O}$ is a primitive quaternion with $\mathrm{Nm}(\beta)=p^{2l+s-2t}$, and since by assumption $\mathcal{O}$ contains a $\xi$-primary class set, 
there exists a unique $\varepsilon\in\mathcal{O}^\times$ (up to sign if $2\in\xi\mathcal{O}$) such that $\beta':=\beta\varepsilon^{-1}$ is a primitive and $\xi$-primary quaternion having the same norm as $\beta$. 
Therefore we can apply Theorem \ref{Zerlegungssatz} to $\beta'$ and obtain the result with $n:=t-l$.
\end{proof}

Since $p\nmid D$, we can take a $p$-adic matrix immersion $\Phi_{p}: H\hookrightarrow\mathrm{M}_{2}(\mathbb{Q}_{p})$.
We now consider the following groups introduced in section \ref{Sec1}:
$$\Gamma_{p}:=\Phi_p(\mathcal{O}[1/p]^\times)/\mathbb{Z}[1/p]^\times\quad\mbox{and}\quad\Gamma_{p,+}:=\Phi_p(\mathcal{O}[1/p]^\times_+)/\mathbb{Z}[1/p]^\times.$$
These are the important groups arising in the $p$-adic uniformisation of a Shimura curve with discriminant $Dp$. We will show that, under certain assumptions on $H$ and $\mathcal{O}$, we can find a Schottky group $\Gamma^{\mathrm{Sch}}\subseteq\Gamma_{p}$ and a finite system of generators for $\Gamma^{\mathrm{Sch}}$. As suggested by the proof of Theorem \ref{teo_estructura_grupos}, we can look for this  group among the \emph{principal congruence subgroups} of $\Gamma_{p}$.

Take $\xi\in\mathcal{O}$ satisfying the right-unit property in $\mathcal{O}$ (cf. Definition \ref{units_property}). From now on we will assume that $2\in\xi\mathcal{O}$. We then distinguish two non-trivial cases: either $\xi=2$ or $\mathrm{Nm}(\xi)=2$ (cf. Remark \ref{xi2}). In the first case, we consider the following group homomorphism:
$$\pi: \mathcal{O}[1/p]^\times \rightarrow  (\mathcal{O}/2\mathcal{O})^\times=(\mathcal{O}/2\mathcal{O})_r^\times,
\quad\alpha \mapsto \alpha+2\mathcal{O}$$
(cf. Remark \ref{xi=n} for the equality). If $\mathrm{Nm}(\xi)=2$, we consider the following group homomorphism:
$$\pi: \mathcal{O}[1/p]^\times \stackrel{\pi'}{\rightarrow} (\mathcal{O}/2\mathcal{O})^\times\stackrel{\lambda_r}{\twoheadrightarrow} (\mathcal{O}/\xi\mathcal{O})_r^\times
\quad\alpha\mapsto \alpha + 2\mathcal{O} \mapsto  \alpha+\xi\mathcal{O},$$
where $(\mathcal{O}/\xi\mathcal{O})_r^\times$ has the structure of multiplicative group induced by $\mathcal{O}^{\times}/\mathbb{Z}^{\times}\simeq(\mathcal{O}/\xi\mathcal{O})_r^\times$ (cf. Remark \ref{group_structure}).
Note that, since $p$ is odd, any $\alpha\in \mathcal{O}[1/p]^\times$ has norm $\mathrm{Nm}(\alpha)\equiv 1\;\mathrm{mod}\;2$ and $\pi$ is well-defined.
Moreover, $\pi$ factorises via $\mathbb{Z}[1/p]^\times$, since $\alpha=\pm p^n\beta$ implies that $\alpha\equiv\beta$ mod 2$\mathcal{O}$.

\begin{lemma}\label{exact_seq}
	Let $\xi\in\mathcal{O}$ be an integral quaternion satisfying the right-unit property in $\mathcal{O}$ with $2\in\xi\mathcal{O}$, and let $p$ be an odd prime. Then 
	there is a split short exact sequence of groups
	$$1\rightarrow\{\alpha\in\mathcal{O}[1/p]^\times\mid\,\alpha-1\in\xi\mathcal{O}\}/\mathbb{Z}[1/p]^\times\rightarrow\mathcal{O}[1/p]^{\times}/\mathbb{Z}[1/p]^\times\stackrel{\pi}{\rightarrow}(\mathcal{O}/\xi\mathcal{O})_r^{\times}\rightarrow 1.$$
\end{lemma}

\begin{proof}
Since the order $\mathcal{O}[1/p]\subseteq H$ satisfies Eichler's condition, we can apply the Strong Approximation Theorem (cf. \cite[Ch III, Th\'{e}or\`{e}me 4.3]{Vigneras1980} and \cite[Lemma 1.1]{Rapinchuk2013}) in order to show that the map $\pi$ is surjective. Finally, since we have a bijection
$\varphi:\mathcal{O}^{\times}/\mathbb{Z}^{\times}\rightarrow(\mathcal{O}/\xi\mathcal{O})_r^{\times},\ \varphi(u)=[u],$
there is a split given by
$(\mathcal{O}/\xi\mathcal{O})_r^\times\simeq\mathcal{O}^\times/\mathbb{Z}^\times\hookrightarrow\mathcal{O}[1/p]^\times/\mathbb{Z}[1/p]^\times$. 
\end{proof}

For any $\xi\in\mathcal{O}$ satisfying the right-unit property in $\mathcal{O}$ with $2\in\xi\mathcal{O}$, we define the following congruence subgroups:
\[
\widetilde{\Gamma}_{p}(\xi):=\Phi_{p}(\{\alpha\in\mathcal{O}[1/p]^{\times}\mid\,\alpha\equiv1\mbox{ mod }\xi\mathcal{O}\})\subseteq\mathrm{GL}_{2}(\mathbb{Q}_{p}),
\]
\[
\Gamma_{p}(\xi):=\widetilde{\Gamma}_{p}(\xi)/\mathbb{Z}[1/p]^{\times}\subseteq\mathrm{PGL}_{2}(\mathbb{Q}_{p}).
\]
 Note that $\Gamma_{p}(\xi)$ is a normal subgroup in $\Gamma_{p}$ of finite index, which can be referred to as the \emph{principal congruence subgroup} of $\Gamma_p$ of level $\xi\mathcal{O}$.
In what follows we will see that, under certain assumptions, the group $\Gamma_{p}(\xi)$ is a Schottky group and we will find a finite and free system of generators for it. With this aim, we define the following finite set:
$$\widetilde{S}:=\Phi_{p}(\{\alpha\in\mathcal{O}\mid\, \mathrm{Nm}(\alpha)=p,\alpha\equiv1\mbox{ mod }\xi\mathcal{O}\})\subseteq\widetilde{\Gamma}_p(\xi),$$
and we denote by $S$ the image of $\widetilde{S}$ in $\Gamma_p(\xi)$ under the natural projection $\widetilde{\Gamma}_p(\xi)\twoheadrightarrow\Gamma_p(\xi)$.

\begin{remark}
	Take $\alpha=a_0+a_1\theta_{1}+a_2\theta_{2}+a_3\theta_{3}\in\mathcal{O}$, where $\{1,\theta_{1},\theta_{2},\theta_3\}$ is an integral basis of $\mathcal{O}$. Suppose that $\Phi_p(\alpha)\in\widetilde{S}$. Then $\alpha=(a_0,a_1,a_2,a_3)\equiv(1,0,0,0)\mbox{ mod }\xi\mathcal{O}$, so $-\alpha=(-a_0,-a_1,-a_2,-a_3)\equiv (1,0,0,0)\mbox{ mod }\xi\mathcal{O}$ and $\overline{\alpha}=(a_0,-a_1,-a_2,-a_3)\equiv(1,0,0,0)\mbox{ mod }\xi\mathcal{O}$, because we are assuming that $2\in\xi\mathcal{O}$.
	Thus, if $\Phi_p(\alpha)\in\widetilde{S}$, then $\Phi_p(\pm\alpha),\Phi_p(\pm\overline{\alpha})\in\widetilde{S}$.
\end{remark}

Following the same idea from \cite[Ch. IX]{Gerritzen_vanderPut1980}, we start by splitting $\widetilde{S}$ in two disjoint sets
$$\widetilde{S}=\Phi_{p}(\{\pm\alpha_1,\ldots,\pm\alpha_s,\pm\overline{\alpha}_1,\ldots,\pm\overline{\alpha}_s\})\cup\Phi_{p}(\{\pm\beta_1,\ldots,\pm\beta_t\}),$$
where $\pm\alpha_i,\pm\overline{\alpha}_{i}\in\mathcal{O}$, for $1\leq i\leq s$, are the \textit{impure} quaternions in $\widetilde{S}$ (i.e. $\mathrm{Tr}(\alpha_{i})\neq 0$), and $\pm\beta_j$, for $1\leq j\leq t$, are the \textit{pure} quaternions in $\widetilde{S}$ (i.e. $\mathrm{Tr}(\beta_{i})=0$).
We then have:
$$S=\{[\Phi_p(\alpha_1)],\ldots,[\Phi_p(\alpha_s)],[\Phi_p(\beta_1)],\ldots,[\Phi_p(\beta_t)]\},$$
where $[\Phi_{p}(\alpha_{i})],[\Phi_{p}(\beta_{j})]$ denote the classes of the matrices $\Phi_{p}(\alpha_{i}),\, \Phi_{p}(\beta_{j})$ inside $\mathrm{PGL}_{2}(\mathbb{Q}_{p})$.
From now on $s$ and $t$ denote the integers such that $\# \widetilde{S}=4s+2t$ and $\#S=s+t$.

\begin{defn}\label{null-trace}
	For every odd prime $p\nmid DN$, we define the following integer:
	$$t_{\xi}(p):=\#\{\alpha\in \mathcal{O}\mid\mathrm{Nm}(\alpha)=p,\,\alpha\equiv 1\;\mathrm{mod}\;\xi\mathcal{O},\ \mathrm{Tr}(\alpha)=0\}.$$
	Therefore, $t_{\xi}(p)=0$ if and only if there is no transformation $\gamma\in S$ with $\mathrm{Tr}(\gamma)=0$. In this case, we will say that $p$ satisfies the \textit{null-trace condition} with respect to $\xi\mathcal{O}$.
\end{defn}

\begin{teo}\label{Theorem_Schottky_group}
Let $H$ be a definite quaternion algebra of discriminant $D$ and $\mathcal{O}=\mathcal{O}_H(N)\subseteq H$ an Eichler order of level $N$ with $h(D,N)=1$. Let $p\nmid DN$ be an odd prime.
Consider
$$S=\{[\Phi_p(\alpha_1)],\ldots,[\Phi_p(\alpha_s)],[\Phi_p(\beta_1)],\ldots,[\Phi_p(\beta_t)]\}\subseteq\mathrm{PGL}_2(\mathbb{Q}_p),$$
where $\alpha_1,\ldots,\alpha_s\in\mathcal{O}$ are the impure quaternions in $\mathcal{O}$ with norm $p$ $\mathrm{(}$up to sign and conjugation$\mathrm{)}$ and $\beta_1,\ldots,\beta_t$ are the pure quaternions in $\mathcal{O}$ with norm $p$ $\mathrm{(}$up to sign$\mathrm{)}$.
Then for every $\xi\in\mathcal{O}$ with $2\in\xi\mathcal{O}$ that satisfies the right-unit property in $\mathcal{O}$, the principal congruence subgroup $\Gamma_p(\xi)$ has $S$ as a system of generators and the only relations between them are
$
\mbox{$[\Phi_p(\beta_i)]^{2}=1$, for  $1\leq i\leq t$}.
$
In particular, if $p$ satisfies the null-trace condition with respect to $\xi\mathcal{O}$, then $\Gamma_{p}(\xi)$ is a Schottky group of rank $s$.
\end{teo}

\begin{proof}
On the one hand, the split short exact sequence described in Lemma \ref{exact_seq} gives an isomorphism of groups
$\Gamma_p\simeq\Gamma_p(\xi)\rtimes(\mathcal{O}/\xi\mathcal{O})_r^\times.$
On the other hand, since we are in the hypothesis of Lemma \ref{isom} we have that $\mathcal{P}=\{[1]\}$ is a $\xi$-primary class set for $\mathcal{O}$. By Theorem \ref{Dec}, every element $\alpha\in\mathcal{O}[1/p]^\times$ can be written uniquely up to sign as $\alpha=p^{n}\prod_{i=1}^r\beta_i\cdot\varepsilon$, for $n\in\mathbb{Z},\ \varepsilon\in\mathcal{O}^\times$ and $\beta_1,\ldots,\beta_r\in\mathcal{O}$ primitive $\xi$-primary quaternions with respect to $\mathcal{P}$ with $\mathrm{Nm}(\beta_i)=p$ and with no factor of the form $\beta_i\cdot\beta_{i+1} = p$. This gives the following split short exact sequence:
$$\begin{array}{ccccccccc}
	1 & \rightarrow & \langle\{\alpha\in\mathcal{O}\mid \mathrm{Nm}(\alpha)=p,\, \alpha\equiv1\mbox{ mod }\xi\mathcal{O}\}\rangle/\mathbb{Z}^\times & \rightarrow \mathcal{O}[1/p]^\times/\mathbb{Z}^\times & \rightarrow & \mathcal{O}^\times/\mathbb{Z}^\times & \rightarrow & 1, \\
	& & & \alpha & \mapsto & \varepsilon
\end{array}$$
since 
$
\{\alpha=p^{n}\prod_{i=1}^{r}\beta_{i}\mid\mathrm{Nm}(\beta_{i})=p,\:\beta_{i}\equiv 1\;\mathrm{mod}\;\xi\mathcal{O}\}= 
\langle\{\alpha\in\mathcal{O}\mid \mathrm{Nm}(\alpha)=p,\, \alpha\equiv1\mbox{ mod }\xi\mathcal{O}\}\rangle.
$
Hence, we obtain the following semi-direct product:
\[
\mathcal{O}[1/p]^{\times}/\mathbb{Z}^\times\simeq
\langle\{\alpha\in\mathcal{O}\mid \mathrm{Nm}(\alpha)=p,\, \alpha\equiv1\mbox{ mod }\xi\mathcal{O}\}\rangle/\mathbb{Z}^\times
\rtimes\mathcal{O}^{\times}/\mathbb{Z}^\times.
\]
After taking the quotient by its centre we obtain
$\Gamma_p=\langle S\rangle\rtimes\mathcal{O}^\times/\mathbb{Z}^\times.$
Since $\xi$ satisfies the right-unit property in $\mathcal{O}$, the natural map $\varphi:\mathcal{O}^\times/\mathbb{Z}^\times\rightarrow(\mathcal{O}/\xi\mathcal{O})_r^\times$ is an isomorphism of groups (cf. Definition \ref{units_property}). Thus the following diagram commutes:
$$\begin{array}{ccccccccc}
	1 & \rightarrow & \Gamma_p(\xi) & \rightarrow & \Gamma_p & \rightarrow & \mathcal{O}^\times/\mathbb{Z}^\times & \rightarrow & 1 \\
	& & \hookuparrow & & \rotatebox{90}{\(=\)} & & \rotatebox{90}{\(\simeq\)} \\
		1 & \rightarrow & \langle S \rangle & \rightarrow & \Gamma_p & \rightarrow & (\mathcal{O}/\xi\mathcal{O})_r^\times & \rightarrow & 1 \\
\end{array}$$
and we obtain that
$
\Gamma_{p}(\xi)=\langle S\rangle.
$
With this we have a finite system of generators for $\Gamma_p(\xi)$.

Now let us look at the possible relations between these generators. First note that for $1\leq i\leq t$ we have the relations $[\Phi_{p}(\beta_{i})]^{2}=1$ in $S$.
Let us see that these are the only possible ones. Since $\mathrm{Nm}(\alpha_i)=p$ for every $1\leq i\leq s$, if there is a relation between the $[\Phi_p(\alpha_i)]$, this must have an even number of elements.
Suppose that $[\Phi_p(\alpha_1)]\cdot\ldots\cdot[\Phi_p(\alpha_{2k})]=[1]$, for $k\geq1$.
This means $\alpha:=\alpha_1\cdot\ldots\cdot\alpha_{2k}=\pm p^m$, for $m\geq1$, and taking norms we have $m=k$. By Theorem \ref{Dec}, $\alpha$ has a decomposition, unique up to sign, as a product $\alpha=p^n\cdot\varepsilon\cdot \prod_{i=1}^r\beta_i$, with $n\in\mathbb{Z},\ \varepsilon\in\mathcal{O}^\times$ and
\begin{enumerate}[(a)]
	\item $\beta_1,\ldots,\beta_r\in\mathcal{O}$ primitive and $\xi$-primary quaternions with respect to $\mathcal{P}=\{[1]\}$,
	\item $\mathrm{Nm}(\beta_i)=p$,
	\item no factor of the form $\beta_i\cdot\beta_{i+1} = p$ appears in the product.
\end{enumerate}
Since $\alpha=\pm p^k$, by the uniqueness of this decomposition we cannot have $\alpha=\alpha_1\cdot\ldots\cdot\alpha_{2k}$ with the $\alpha_i$ satisfying conditions (a)-(c) simultaneously. Since the $\alpha_i$ already satisfy conditions (a) and (b), this means that there are $\alpha_i,\alpha_{i+1}$ such that $\alpha_i\alpha_{i+1}=p$, so $\alpha_i=\overline{\alpha}_{i+1}$, and this is a contradiction because we already excluded these elements in $S$.
\end{proof}

Finally, when $\Gamma_p(\xi)$ is a Schottky group (i.e. when $p$ satisfies $t_{\xi}(p)=0$), we want to know the rank of this group; that is, we want to compute $s$. Before doing so, we need some notation.

Once a presentation $H=\left(\frac{a,b}{\mathbb{Q}}\right)$ and a basis for $\mathcal{O}$ are fixed, let $\mathrm{N}_{(a,b),4}(X,Y,Z,T):=X^2-aY^2-bZ^2+abT^2$ denote the quaternary normic form associated with $H$ with respect to the basis $\{1,i,j,ij\}$ and $\mathrm{N}_{\mathcal{O},4}(X,Y,Z,T)$ the normic form associated with $\mathcal{O}$ with respect to the chosen basis. We consider the number of representations of a prime $p$ by these normic forms:
$$r(\mathrm{N}_{(a,b),4},p;\mathbb{Z}):=\#\{(x,y,z,t)\in\mathbb{Z}^4\mid\mathrm{N}_{(a,b),4}(x,y,z,t)=p\},$$
$$r(\mathrm{N}_{\mathcal{O},4},p;\mathbb{Z}):=\#\{(x,y,z,t)\in\mathbb{Z}^4\mid\mathrm{N}_{\mathcal{O},4}(x,y,z,t)=p\}.$$
Note that $r(\mathrm{N}_{\mathcal{O},4},p;\mathbb{Z})$ does not depend on the presentation of the algebra nor on the basis of $\mathcal{O}$ (cf. \cite{AlsinaBayer2004} Prop. 3.86), while $r(\mathrm{N}_{(a,b),4},p;\mathbb{Z})$ does depend on the presentation of $H$. Given $\xi\in\mathcal{O}$ we also define
$$r(\mathrm{N}_{\mathcal{O},4},p;\xi):=\#\{\alpha\in\mathcal{O}\mid \mathrm{Nm}(\alpha)=p,\ \alpha\equiv 1\mbox{ mod }\xi\mathcal{O}\}.$$
Note that, when $\xi=2$, then
$r(\mathrm{N}_{\mathcal{O},4},p;2)=\#\{(x,y,z,t)\in\mathbb{Z}^4\mid\mathrm{N}_{\mathcal{O},4}(x,y,z,t)=p,\ (x,y,z,t)\equiv(1,0,0,0)\mbox{ mod } 2\},$
and when $\xi=1$, then
$r(\mathrm{N}_{\mathcal{O},4},p;1)=r(\mathrm{N}_{\mathcal{O},4},p;\mathbb{Z}).$

The following is a fundamental result concerning the numbers of representations of a prime $p$ by certain quaternary quadratic forms that generalises a well-known result of \cite{Hurwitz1896}.

\begin{teo}\label{primary_representations}
	Let $H$ be a definite quaternion algebra of discriminant $D$ and let $\mathcal{O}\subseteq H$ be an Eichler order of level $N$ with $h(D,N)=1$. 
	Let $p\nmid DN$ be a prime, and let $\xi\in\mathcal{O}$ be an integral quaternion such that $\mathcal{O}/\xi\mathcal{O}$ contains a $\xi$-primary class set. 
	Then the cardinality of the finite set
	$
	\{\pi\in\mathcal{O}\mid\mathrm{Nm}(\pi)=p,\ \pi\mbox{ is }\xi\mbox{-primary}\,\}
	$
is equal to:
	\begin{enumerate}[$(a)$]
		\item
		$p+1$, if $2\notin\xi\mathcal{O}$,
		\item
		$2(p+1)$, if $2\in\xi\mathcal{O}$.
	\end{enumerate}
\end{teo}

\begin{proof}
	Let $c:=\#\{\pi\in\mathcal{O}\ \xi\mbox{-primary}\mid\,\mathrm{Nm}(\pi)=p\}$.
	Since $p\nmid DN$ we have an isomorphism of algebras $\mathcal{O}/p\mathcal{O}\simeq\mathrm{M}_{2}(\mathbb{Z}/p\mathbb{Z})$.  
	Therefore 
	$
	A_{p}:=\{[\alpha]\in\mathcal{O}/p\mathcal{O}\mid\,\alpha\notin p\mathcal{O},\;\mathrm{Nm}(\alpha)\equiv 0\;\mathrm{mod}\ p\}
	$
has $(p^{2}-1)(p+1)$ elements.
	We now distinguish two cases.
	\begin{enumerate}[($a$)]
	\item
	If $2\notin\xi\mathcal{O}$, we define the map
	$
	\Sigma : A_{p}\rightarrow \{\pi\in\mathcal{O}\mid\,\mathrm{Nm}(\pi)=p,\;\xi\mbox{-primary}\}
	$
	assigning to $[\alpha]\in A_{p}$ the unique quaternion $\pi_{[\alpha]}$ which is $\xi$-primary and such that $\alpha\mathcal{O}+p\mathcal{O}=\pi_{[\alpha]}\mathcal{O}$.
	\item
	If $2\in\xi\mathcal{O}$, the primary quaternion $\pi_{[\alpha]}$ defined above is unique only up to sign, so the same assignment gives a well-defined map 
	$
	\Sigma : A_{p}\rightarrow \{\pi\in\mathcal{O}\mid\,\mathrm{Nm}(\pi)=p,\;\xi\mbox{-primary}\}/\{\pm 1\},
	$
	where the quotient by $\{\pm 1\}$ denotes the equivalence relation identifying $\pi$ and $-\pi$.
	\end{enumerate}
	Finally, we will show that: 
	\begin{enumerate}[($a$)]
	\item
	if $2\notin\xi\mathcal{O}$, then $c=\# A_{p}/(p^{2}-1)=p+1$,
	\item
	if $2\in\xi\mathcal{O}$, then $c=2(\# A_{p}/(p^{2}-1))=2(p+1)$.
	\end{enumerate}
	Indeed, the map $\Sigma$ is, in both cases, a $1$ to $p^{2}-1$ correspondence. To prove this, let us see that 
	$\#\Sigma^{-1}(\Sigma([\alpha]))=p^{2}-1$. First we observe that 
	$
	\Sigma^{-1}(\pi_{[\alpha]})=\{[\alpha]\lambda\in A_{p}\mid\, \lambda\in\mathcal{O}/p\mathcal{O}\},
	$. Thus
	$
	\#(\mathcal{O}/p\mathcal{O})/\#\{[x]\in \mathcal{O}/p\mathcal{O}\mid\, \alpha x\equiv 0\;\mathrm{mod}\;\xi\mathcal{O} \}-1=p^{4}/p^{2}-1=p^{2}-1.
	$
\end{proof}

\begin{remark}
	Note that, when $D,N,H,\mathcal{O}$ and $\xi\in\mathcal{O}$ are taken as in Table \ref{Table_xi}, this result says that the number of representations 
$r(\mathrm{N}_{\mathcal{O},4},p;\xi)$ is equal to $2(p+1)$. 
	One can then generalise Hurwitz's result about the number of representations of a prime $p$ as a sum of 4-squares to the definite quaternion algebras of Table \ref{Table_xi}.
	In fact, when $D=2,\ N=1$ and $\xi = 2$, it is easy to prove that $r(\mathrm{N}_{\mathcal{O},4},p;\xi) = 4\,r(\mathrm{N}_{H,4},p) = 8(p+1),$ where $\mathrm{N}_{H,4} = \mathrm{N}_{(-1,-1),4}(X,Y,Z,T)=X^2+Y^2+Z^2+T^2.$
	Let us also consider the case $D = 3,N= 1,$ and $\xi = 2$. After Lemma \ref{lemma_D3} it is clear that $r(\mathrm{N}_{\mathcal{O},4},p;\xi) = 2\,r(\mathrm{N}_{H,4},p) = 4(p+1)$, where  $\mathrm{N}_{H,4} = \mathrm{N}_{(-1,-3),4}(X,Y,Z,T)=X^2+Y^2+3Z^2+3T^2$ (as Liouville proved in \cite{Liouville1860}).
\end{remark}

\begin{cor}\label{rank}
	Let $H$ be a definite quaternion algebra of discriminant $D$, let $\mathcal{O}=\mathcal{O}_H(N)\subseteq H$ be an Eichler order of level $N$, and let $p\nmid DN$ be an odd prime. Assume that:
\begin{enumerate}[$(i)$]
	\item 
	$h(D,N)=1$,
	\item 
	there exists $\xi\in\mathcal{O}$ that satisfies the right-unit property in $\mathcal{O}$ with $2\in\xi\mathcal{O}$ $($cf. Table \ref{Table_xi}$)$,
	\item 
	the prime $p$ satisfies the null-trace condition with respect to $\xi\mathcal{O}$ $($cf. Definition $\ref{null-trace})$.
\end{enumerate}
	Then $\Gamma_p(\xi)$ is a Schottky group of rank
	$\mathrm{rank}\,\Gamma_p(\xi)=\frac{p+1}{2}.$
\end{cor}

\begin{proof}
	After Theorem $\ref{Theorem_Schottky_group}$
	we have that $\Gamma_{p}(\xi)$ is a Schottky group of rank $\# S=s$. Hence we need to compute $s=\frac{1}{4}r(\mathrm{N}_{\mathcal{O},4},p;\xi)$ and by Theorem \ref{primary_representations} we find $s=(p+1)/2$.
\end{proof}

\subsection{The null-trace condition}\label{null-trace_condition}
	Here we want to show that the set of primes satisfying condition ($iii$) of Corollary \ref{rank} is actually infinite, so this will yield to infinite families of Shimura curves. 
	We will take a close look at the case $D=3,N=1$ and rewrite the null-trace condition in a more amenable and arithmetical way.

\begin{lemma}\label{lemma_D3}
A quaternion $\alpha=a_{0}+a_{1}i+a_{2}j+a_{3}k\in H=\left(\frac{-1,-3}{\mathbb{Q}}\right)$ belongs to the finite set 
\[
A:=\{\alpha\in\mathcal{O}\mid\,\alpha\equiv 1\ \mathrm{mod}\ 2\mathcal{O},\;\mathrm{Nm}(\alpha)=p\}
\]
if and only if it satisfies the following conditions:
\begin{enumerate}[$(i)$]
\item
$a_{i}\in\mathbb{Z}$ for every $0\leq i\leq 3$,
\item
$a_{0}^{2}+a_{1}^{2}+3a_{2}^{2}+3a_{3}^{2}=p$,
\item
$a_{0}+a_{3}\equiv 1\;\mathrm{mod}\,2$, $a_{1}+a_{2}\equiv 0\;\mathrm{mod}\,2$.
\end{enumerate}

\end{lemma}

\begin{proof}
Let $H=\left(\frac{-1,-3}{\mathbb{Q}}\right)$ be the definite quaternion algebra of discriminant $D=3$ and consider the maximal order $\mathcal{O}\subseteq H$ with basis $\{1,i,\lambda,\mu\}$, where $\lambda:=(i+j)/2$ and $\mu:=(1+k)/2$. If we take $\alpha=a_0+a_1i+a_2j+a_3k\in A$ and express it in the integral basis of $\mathcal{O}$, namely $\alpha=x+yi+z\lambda+t\mu$, we obtain the following relations:
	$$a_0=x+t/2,\quad a_1=y+z/2,\quad a_2=z/2,\quad a_3=t/2.$$
	The condition $\alpha\equiv1\mbox { mod }2$ tells us that $z\equiv t\equiv 0\mbox{ mod }2$, so $a_i\in\mathbb{Z}$ for every $0\leq i\leq 3$. Moreover we also have $a_0+a_3=x+t\equiv 1\mbox{ mod }2$ and $a_1+a_2=y+z\equiv0\mbox{ mod }2$, so we find the conditions of the statement on the coefficients $a_{i}$. The converse is clearly true.
\end{proof}

\begin{prop}
Let $H$ be the definite quaternion algebra of discriminant $D=3$ and $\mathcal{O}\subseteq H$ a maximal order. With the notations from Theorem $\ref{Theorem_Schottky_group}$ we have that, if $p\equiv 1\,\mathrm{mod}\,4$, then 
\[
t_{2}(p):=\#\{\alpha\in\mathcal{O}\mid\,\alpha\equiv 1\ \mathrm{mod}\ 2\mathcal{O},\;\mathrm{Nm}(\alpha)=p,\,\mathrm{Tr}(\alpha)=0\}=0.
\] 
\end{prop}

\begin{proof}	
	Suppose that there exists a pure quaternion $\alpha\in\mathcal{O}$ of norm $p$ such that $\alpha\equiv 1\,\mathrm{mod}\,2\mathcal{O}$. Then the previous conditions on the coefficients $a_i\in\mathbb{Z}$ translate into
	$a_3\equiv1\mbox{ mod }2,\  a_1+a_2\equiv0\mbox{ mod }2$ and $a_1^2+3a_2^2+3a_3^2=p.$
	If we now reduce modulo 4 we obtain
	$p=a_1^2+3a_2^2+3a_3^2\equiv 4a_2^2+3\equiv 3\;\mathrm{mod}\,4.$
\end{proof}

Similarly one can prove the following result for the Hurwitz quaternions (\cite[Ch. IX]{Gerritzen_vanderPut1980}).

\begin{prop}
Let $H$ be the definite quaternion algebra of discriminant $D=2$ and let $\mathcal{O}$ be a maximal order in $H$. If $p\equiv 1\,\mathrm{mod}\,4$, then $t_2(p)=0$.
\end{prop}

\subsection{$p$-adic fundamental domains and their reduction-graphs}

In this section we describe how to construct good fundamental domains for Mumford curves uniformised by the Schottky groups of the previous section.
By \cite[Satz 1]{Gerritzen1974} we know that one can always find a system of generators $S=\{\gamma_{1},\dots,\gamma_{g}\}$ for a Schottky group $\Gamma$ such that a good fundamental domain for $\Gamma$ with respect to $S$ exists (cf. Definition \ref{dom_fund_grupos_Schottky}). One says that a system of generators of this kind is \emph{in good position}. Finding generators in good position for a given Schottky group is an important problem, which is computationally solved in \cite{MorrisonRen2015}. 

In the next result we give sufficient conditions for a system of generators of a Schottky group to be in good position. This applies to the generators arising from Theorem \ref{Theorem_Schottky_group} and will allow us to describe the stable reduction-graph of Mumford curves covering $p$-adic Shimura curves.

\begin{prop}\label{prop_dom_fund}
Let $\Gamma\subseteq\mathrm{PGL}_{2}(\mathbb{Q}_{p})$ be a cocompact Schottky group of rank $g$, and let $S=\{\gamma_{1},\dots,\gamma_{g}\}$ be a system of generators. Assume that:
\begin{enumerate}[$(a)$]
\item
The rank of $\Gamma$ is $g=(p+1)/2$.
\item
$\mathrm{Det}(\gamma_{i})=p\mathbb{Q}_{p}^{\times2}$ for every $i=1,\dots,g$.
\end{enumerate}
Then
$\mathbb{P}^{1}(\mathbb{C}_{p})\smallsetminus\left(\cup_{a=0}^{p-1}\mathbb{B}^{-}(a,1/\sqrt{p})\cup \mathbb{B}^{-}(\infty,1/\sqrt{p})\right)$
is a good fundamental domain for the action of $\Gamma$ with respect to $S$.
\end{prop}

\begin{proof}
For every $i=1,\dots,g$, denote by $\alpha_{i},\alpha_{i+g}\in\mathbb{P}^{1}(\mathbb{Q}_{p})$ the two fixed points of the transformations $\gamma_{i},\gamma_{i}^{-1}$, and define the following admissible subdomain of $\mathcal{H}_{p}$:
\[
\mathcal{F}_{\Gamma}:=\mathbb{P}^{1}(\mathbb{C}_{p})\smallsetminus\bigcup_{i=1}^{2g}\mathbb{B}^{-}\left(\alpha_{i},1/\sqrt{p}\right).
\]
Condition ($i$) of Definition \ref{dom_fund_grupos_Schottky} is clearly satisfied, since the transformations $\gamma_{i}$ are hyperbolic, so their fixed points are in $\mathbb{P}^{1}(\mathbb{Q}_{p})$. Moreover, since these points satisfy that $\alpha_{i}=\lim_{n\rightarrow\infty}(\gamma_{i}^{n}\cdot v^{0})$ and $\alpha_{i+g}=\lim_{n \rightarrow\infty}(\gamma_{i}^{-n}\cdot v^{0})$, we have
$
\gamma_{i}\cdot e^{(1)}_{\widetilde{\alpha}_{i}}=-e_{\widetilde{\alpha}_{i+g}}^{(1)},
$
and using the reduction map of Theorem \ref{red_map} we find that condition ($iii$) is also satisfied. Thus we only need to check condition ($ii$).
	So suppose that there exist $i,j\in\{1,\dots,(p+1)/2\}, i\neq j,$ such that 
$\mathbb{B}^{+}(\alpha_{i},1/\sqrt{p})\cap \mathbb{B}^{+}(\alpha_{j},1/\sqrt{p})\neq \emptyset.$
	Since these two balls have the same radius we obtain $\mathbb{B}^{+}(\alpha_{i},1/\sqrt{p})= \mathbb{B}^{+}(\alpha_{j},1/\sqrt{p})$. 
Let $\widetilde{\alpha}_{i},\widetilde{\alpha}_{j}\in\mathbb{P}^{1}(\mathbb{F}_{p})$ be the corresponding reductions mod $p$.
Applying the reduction map of Theorem \ref{red_map} and the fact that this is equivariant with respect to the action of $\mathrm{PGL}_{2}(\mathbb{Q}_{p})$, we have $\gamma_{i}\cdot e^{(1)}_{\widetilde{\alpha}_{i}}=-e^{(1)}_{\widetilde{\alpha}_{i+g}}$ and $\gamma_{j}\cdot e^{(1)}_{\widetilde{\alpha}_{j}}=-e^{(1)}_{\widetilde{\alpha}_{j+g}}$.
	Therefore $\gamma_{i}\gamma_{j}^{-1}\cdot(-e^{(1)}_{\widetilde{\alpha}_{j+g}})=-e^{(1)}_{\widetilde{\alpha}_{i+g}}$, and the transformation $\gamma_{i}\gamma_{j}^{-1}$ fixes a vertex of $\mathcal{T}_{p}$, which is a contradiction since the group $\Gamma_{p}$ is torsion-free.

Finally the subdomain $\mathcal{F}_{\Gamma}$ is a good fundamental domain for $\Gamma$ with respect to the system of generators $S$ and, by Lemma \ref{cns_bolas_disjuntas} ($b$), this is equal to the subdomain of the statement.
\end{proof} 

\begin{teo}\label{Theorem_Mumford_curve}
Let $H$ be a definite quaternion algebra of discriminant $D$, $\mathcal{O}\subseteq H$ an Eichler order of level $N$, and $p\nmid DN$ an odd prime. Under the assumptions of Corollary $\ref{rank}$, we have:
\begin{enumerate}[$(a)$]
\item
The group $\Gamma_{p}(\xi)$ is a Schottky group of rank $(p+1)/2$ generated by the transformations in $\mathrm{PGL}_{2}(\mathbb{Q}_{p})$ represented by the matrices in
$\widetilde{S}:=\Phi_{p}(\{\alpha\in\mathcal{O}\mid\,\mathrm{Nm}(\alpha)=p,\alpha \equiv 1\;\mathrm{mod}\;\xi\mathcal{O}\}).$
\item
A good fundamental domain for the action of $\Gamma_{p}(\xi)$ with respect to the system of generators $\widetilde{S}$ is the admissible subdomain of $\mathcal{H}_{p}$
\[
\mathcal{F}_{p}(\xi):=\mathbb{P}^{1,rig}(\mathbb{C}_{p})\smallsetminus \bigcup_{a\in\{0,\dots,p-1,\infty\}}\mathbb{B}^{-}(a,1/\sqrt{p}).
\]
\item
If $X_{p}(\xi)$ denotes the Mumford curve associated with $\Gamma_{p}(\xi)$, the rigid analytic curve $X_{p}(\xi)^{rig}$ is obtained from $\mathcal{F}_{p}(\xi)$ with the following pair-wise identifications:
for every $\gamma\in\widetilde{S}$,
\[
\gamma\left(\mathbb{P}^{1}(\mathbb{C}_{p})\smallsetminus\mathbb{B}^{-}(\tilde{\alpha}_{\gamma},1/\sqrt{p})\right)=\mathbb{B}^{+}(\tilde{\alpha}_{\gamma^{-1}},1/\sqrt{p}),
\]
\[
\gamma\left(\mathbb{P}^{1}(\mathbb{C}_{p})\smallsetminus\mathbb{B}^{+}(\tilde{\alpha}_{\gamma},1/\sqrt{p})\right)=\mathbb{B}^{-}(\tilde{\alpha}_{\gamma^{-1}},1/\sqrt{p}),
\]
where $\widetilde{\alpha}_{\gamma}$ and $\widetilde{\alpha}_{\gamma^{-1}}$ are defined as the reduction in $\mathbb{P}^{1}(\mathbb{F}_{p})$ of the fixed points of $\{\gamma,\gamma^{-1}\}$.
\item
The stable reduction-graph of $X_{p}(\xi)$ is the quotient of the open subtree of $\mathcal{T}_p$ given by $\mathcal{T}_{p}^{(1)}\smallsetminus\{v^{(1)}_{0},\dots,v^{(1)}_{p-1},v^{(1)}_{\infty}\}$ via the pair-wise identification of the $p+1$ oriented edges given by $\gamma e_{\widetilde{\alpha}_{\gamma}}=-e_{\widetilde{\alpha}_{\gamma^{-1}}}$, for every $\gamma\in \widetilde{S}$.
\end{enumerate}
\end{teo}


\begin{proof}
By Theorem \ref{Theorem_Schottky_group} we know that $\Gamma_{p}(\xi)$ is a Schottky group and by Corollary \ref{rank} it has rank $(p+1)/2$. Therefore (a) follows. Now we are in the hypothesis of Proposition \ref{prop_dom_fund} and ($b$) follows. By the theory of Mumford curves (cf. section \ref{Sec1}), there exists a projective curve $X_{p}(\xi)$ over $\mathbb{Q}_{p}$ such that its rigidification is
$
\Gamma_{p}(\xi)\backslash\mathcal{F}_{p}(\xi) \simeq \Gamma_{p}(\xi)\backslash\mathcal{H}_{p}\simeq X_{p}(\xi)^{rig},
$
and by Definition \ref{dom_fund_grupos_Schottky} we have the identifications of ($c$).
Finally, using the reduction map of Theorem \ref{red_map} we find that $\mathcal{F}_{p}(\xi)$ together with its identifications reduces to a rose-graph with $(p+1)/2$ petals.
\end{proof}

\begin{cor}\label{Corollary}
Consider the Shimura curve $X(Dp,N)$ of discriminant $Dp$, with $(D,N)=1$ and $p\nmid DN$, associated with an Eichler order of level $N$ in the indefinite quaternion algebra of discriminant $Dp$. Let $H$ be the definite quaternion algebra of discriminant $D$ and $\mathcal{O}=\mathcal{O}_H(N)\subset H$ an Eichler order of level $N$. 
If $\mathcal{O}$ satisfies the conditions of Theorem $\ref{Theorem_Mumford_curve}$, then there exists a Mumford curve of genus $(p+1)/2$ covering the $p$-adic rigid Shimura curve $(X(Dp,N)\otimes_{\mathbb{Q}}\mathbb{Q}_{p})^{rig}$ and the degree of the cover is equal to the order of the group $\mathcal{O}^{\times}/\mathbb{Z}^{\times}$.
\end{cor}

\begin{proof}
	Using Theorem \ref{Theorem_Schottky_group} we find a Schottky group $\Gamma_{p}(\xi)$ which is a normal subgroup of $\Gamma_{p}$ of index $\#\Gamma_{p}/\Gamma_{p}(\xi)=\#\mathcal{O}^{\times}/\mathbb{Z}^{\times}$. 
	Therefore, there exists a rigid analytic curve isomorphic to $\Gamma_{p}\backslash\mathcal{H}_{p}$ which is a finite quotient of the rigidification of the Mumford curve $X_{p}(\xi)$ described in Theorem \ref{Theorem_Mumford_curve}. 
	Moreover, by Theorem \ref{CD} we have that $(X(D,N)\otimes_{\mathbb{Q}}\mathbb{Q}_{p})^{rig}\simeq \Gamma_{p}\backslash(\mathcal{H}_{p}\otimes_{\mathbb{Q}_{p}}\mathbb{Q}_{p^2})$. 
\end{proof}

\begin{remarks}[About the conditions of Theorem \ref{Theorem_Mumford_curve}]\label{remarks_Thm}
\hfill
\begin{enumerate}[(1)]
\item
Condition ($ii$) of Theorem \ref{Theorem_Mumford_curve} is satisfied by all the definite Eichler orders in Table \ref{Table_xi}.
\item
If one wants to do explicit computations, condition ($iii$) of Theorem \ref{Theorem_Mumford_curve} is easy to check in each case once $p$ is fixed. Actually, once the finite set of generators for the group $\Gamma_{p}(\xi)$ is computed, one only needs to check whether any of them has null trace. 
\item
When condition ($iii$) of Theorem \ref{Theorem_Mumford_curve} is not satisfied, it is still possible to find a Schottky group $\Gamma_{p}(\xi)^{\star}\subseteq\Gamma_{p}(\xi)$, avoiding the $2$-torsion, together with an explicit system of generators. This is shown in \cite[Ch. IX]{Gerritzen_vanderPut1980} when $D=2,N=1$, and can be extended to all cases considered in the present study. Using Magma, we compute generators 
for $\Gamma_{p}(\xi)^{\star}$ for primes $p\leq 2000$ in all cases in Table \ref{Table_xi}, and check that its rank is always $p$. 
Therefore we find a Mumford curve of genus $p$ covering the $p$-adic Shimura curve.
\end{enumerate}
\end{remarks}

\section{Application of the method to bad reduction of Shimura curves}\label{Sec4}

\subsection{Reduction-graphs with lengths}
For the definitions of \emph{admissible curve} and \emph{reduction-graph with lengths} we refer the reader to \cite{JordanLivne1984}. The Drinfel'd model $\mathcal{X}(Dp,N)$ over $\mathbb{Z}_{p}$ of the Shimura curve $X(Dp,N)$ (cf. Definition \ref{Drinfeld_model}), is then an admissible curve, and its reduction-graph $\Gamma_{p,+}\backslash\mathcal{T}_{p}$ is a graph with lengths. 

In \cite{FrancMasdeu2014} the authors develop an algorithm to compute the reduction-graphs with lengths at $p$ of Shimura curves $X(Dp,N)$ associated with an Eichler order of level $N$ inside an indefinite quaternion algebra of discriminant $Dp$. Thus this problem is completely solved for all Shimura curves. Nevertheless, in this section we give general formulas describing the reduction-graphs with lengths for those families of Shimura curves satisfying Corollary \ref{Corollary}, that is, families $X(Dp,N)$ such that $h(D,N)=1$, except for $X(2p,5)$ and $X(7p,1)$. After Theorem \ref{Theorem_Mumford_curve}, we can recover these formulas thanks to the explicit knowledge of the Mumford curves covering the $p$-adic Shimura curves considered. We are led to study the two following coverings:
\begin{enumerate}[(i)]
\item
$\Gamma_{p}(\xi)\backslash\mathcal{T}_{p}\twoheadrightarrow\Gamma_{p}\backslash\mathcal{T}_{p}$ of degree $\#\mathcal{O}^{\times}/\mathbb{Z}^{\times}$.
\item
$\Gamma_{p,+}\backslash\mathcal{T}_{p}\twoheadrightarrow\Gamma_{p}\backslash\mathcal{T}_{p}$ of degree $2$.
\end{enumerate} 
The reduction-graph with lengths $\Gamma_{p}\backslash\mathcal{T}_{p}$ is, by the way, the reduction-graph of an integral model of the Atkin-Lehner quotient of $X(Dp,N)$ associated with the involution of norm $p$. 

\begin{defn}
	Let $\mathcal{G}$ be a graph and $\Gamma$ a group acting on the left on $\mathcal{G}$. The \emph{length} of a vertex $[v]\in\mathrm{Ver}(\Gamma\backslash\mathcal{G})$ (resp. an edge $[e]\in\mathrm{Ed}(\Gamma\backslash\mathcal{G})$) is the cardinality of the stabiliser $\Gamma_v$ of a representative $v\in\mathrm{Ver}(\mathcal{G})$ (resp. the stabiliser $\Gamma_e$ of a representative $e\in\mathrm{Ed}(\mathcal{G})$) inside the group $\Gamma$. We denote it by $\ell([v])$ (resp. $\ell([e])$).
	
	Given a vertex $v\in\mathrm{Ver}(\mathcal{G})$, we denote by $\mathrm{Star}(v)$ the subset of oriented edges $e\in\mathrm{Ed}(\mathcal{G})$ with origin in $v$.
	We define the integer
	$c_{n}:=\#\{[e]\in\mathrm{Ed}(\Gamma\backslash\mathcal{G})\mid\,\ell([e])=n\}$,
	for every $n\geq1$. 
\end{defn}

If $\Gamma$ is a group acting on the left on a graph $\mathcal{G}$ there is a natural surjection $\mathrm{Star}(v)\rightarrow\mathrm{Star}([v])$. Applying the orbit-stabiliser theorem, one can easily prove the following lemma.

\begin{lemma}\label{formula}
	Let $\mathcal{G}$ be a graph and let $\Gamma$ be a group acting on the left on $\mathcal{G}$. Take $v\in\mathrm{Ver}(\mathcal{G})$ and $e\in\mathrm{Star}(v)$. Then $\ell([e])$ divides $\ell([v])$ and the following formula holds:
	$$\#\mathrm{Star}(v)=\sum_{[e]\in\mathrm{Star}([v])}\frac{\ell([v])}{\ell([e])}.$$
\end{lemma}

For ease exposition, from now on we fix a presentation of $H=\left(\frac{a,b}{\mathbb{Q}}\right)$, and consider in each case the ones in Table \ref{Table_xi}. For every prime $p$ with $\left(\frac{a}{p}\right)=1$, we take the immersion
$$
\Phi_{p}: H  \rightarrow   \mathrm{M}_{2}(\mathbb{Q}_{p}(\sqrt{a}))=\mathrm{M}_{2}(\mathbb{Q}_{p}),\quad
x_{0}+x_{1}i+x_{2}j+x_{3}k \mapsto
\left(\begin{smallmatrix}x_{0}+x_{1}\sqrt{a} & x_{2}+x_{3}\sqrt{a} \\ b(x_{2}-x_{3}\sqrt{a}) & x_{0}-x_{1}\sqrt{a}\end{smallmatrix}\right).
$$

\begin{remark}
In general, once a matrix immersion $\Phi_{p}:H\rightarrow\mathrm{M}_{2}(\mathbb{Q}_{p}(\sqrt{a}))$ is fixed, one can always change the presentation of $H$ in order to have $\Phi_{p}(H)\subseteq\mathrm{M}_{2}(\mathbb{Q}_{p})$ (cf. \cite{SerreCdA1970}).
\end{remark}

\begin{prop}
Let $B$ be an indefinite quaternion algebra of discriminant $Dp$ and $\mathcal{O}_{B}(N)\subseteq B$ an Eichler order of level $N$.
Let $X(Dp,N)$ be the Shimura curve associated with $B$ and $\mathcal{O}_{B}(N)$. 
Let $H$ be the definite quaternion algebra of discriminant $D$ and $\mathcal{O}=\mathcal{O}_{H}(N)\subseteq H$ an Eichler order of level $N$.
Assume that $\mathcal{O}$ satisfies the hypothesis of Theorem $\ref{Theorem_Mumford_curve}$. Then  $\Gamma_{p}\backslash\mathcal{T}_{p}$ has one vertex of length $\#\mathcal{O}^\times/\mathbb{Z}^\times$, and the number of edges is described in Table $\ref{Table_edges_graph}$.
\end{prop}

\begin{table}[h]
\captionsetup{font=footnotesize}
\centering
\scalebox{0.65}{
\begin{tabular}{|c||c|c|c|c||c|c|c||c|c||c|} \hline
	$D$ & 2 & 2 & 2 & 2 & 3 & 3 & 3 & 5 & 5 & 13 \\ \hline
	$N$ & 1 & 3 & 9 & 11 & 1 & 2 & 4 & 1 & 2 & 1 \\ \hline
	$c_1$ & $\frac{1}{12}\left(p-9-4\left(\frac{3}{p}\right)\right)$ & $\frac{1}{3}\left(p-\left(\frac{3}{p}\right)\right)$ & $p+1$ & $p+1$ & $\frac{1}{6}\left(p-6-\left(\frac{3}{p}\right)\right)$ & $\frac{1}{2}(p-1)$ & $p+1$ & $\frac{1}{3}\left(p-\left(\frac{-3}{p}\right)\right)$ & $p+1$ & $p+1$ \\ \hline
	$c_2$ & 1 & 0 & 0 & 0 & 2 & 2 & 0 & 0 & 0 & 0 \\ \hline
	$c_3$ & $1+\left(\frac{3}{p}\right)$ & $1+\left(\frac{3}{p}\right)$ & 0 & 0 & $\frac{1}{2}\left(1+\left(\frac{3}{p}\right)\right)$ & 0 & 0 & $1+\left(\frac{-3}{p}\right)$ & 0 & 0 \\ \hline
\end{tabular}
}
\caption{Number of edges with lengths of the graph $\Gamma_{p}\backslash\mathcal{T}_{p}$}\label{Table_edges_graph}
\vspace{-1.0em}
\end{table}

\begin{proof}
	Since $\Gamma_{p}(\xi)\backslash\mathcal{T}_{p}$ has one vertex of length $1$ (cf. Theorem \ref{Theorem_Mumford_curve}), we have that 
	$\Gamma_p\backslash\mathcal{T}_p\simeq(\Gamma_p(\xi)\backslash\mathcal{T}_p)/(\Gamma_p/\Gamma_p(\xi))$
	has only one vertex $[v^{0}]$ of length $\ell([v^{0}])=\#\Gamma_{p}/\Gamma_{p}(\xi)=\#\mathcal{O}^\times/\mathbb{Z}^\times$.
	After Theorem \ref{Theorem_Mumford_curve} (d), we know that $\Gamma_p(\xi)\backslash\mathcal{T}_p$ consists of $p+1$ oriented edges
	pair-wise identified.
		
	Now we only need to describe the action of $\mathcal{O}^\times/\mathbb{Z}^\times\simeq \Gamma_p/\Gamma_p(\xi)$ on this graph.
	Note that the fundamental domain for the action of $\Gamma_{p}(\xi)$ on $\mathcal{T}_{p}$ 
	is the open tree 
	$
	\mathcal{T}_{p}^{(1)}\smallsetminus\{v_{a}^{(1)}\mid\, a\in\mathbb{P}^{1}(\mathbb{F}_{p})\}\simeq\mathrm{Star}(v^{0}),
	$
	so we can identify each of its open edges with the $\mathbb{F}_{p}$-rational points of $\mathbb{P}^{1}_{\mathbb{F}_{p}}$ by fixing a bijection $\delta:\mathrm{Star}(v^{0})\simeq \mathbb{P}^{1}(\mathbb{F}_{p})$. This gives rise to the following action on the edges:
	$$\mathrm{PGL}_{2}(\mathbb{Z}_{p})\times \mathrm{Star}(v^{0}) \rightarrow\mathrm{Star}(v^{0}),\quad (\gamma,e) \mapsto \mathrm{red}(\gamma)\cdot\delta(e)$$
	where $\mathrm{red}: \mathrm{PGL}_{2}(\mathbb{Z}_{p})\twoheadrightarrow\mathrm{PGL}_{2}(\mathbb{F}_{p})=\mathrm{Aut}(\mathbb{P}^{1}(\mathbb{F}_{p}))$ is the natural projection. Since $\Phi_{p}(\mathcal{O}^{\times})/\mathbb{Z}^{\times}\subseteq \mathrm{PGL}_{2}(\mathbb{Z}_{p})$, in order to study the action of $\Phi_{p}(\mathcal{O}^{\times})/\mathbb{Z}^{\times}$ on $\Gamma_{p}(\xi)\backslash\mathcal{T}_{p}$ we need to study the fixed points in $\mathbb{P}^{1}(\mathbb{F}_{p})$ of these finite order transformations.
		
	Since, by Lemma \ref{formula} the length $\ell([e])$ of an edge $[e]\in\mathrm{Ed}(\Gamma_p\backslash\mathcal{T}_p)$ has to divide $\ell([v^{0}])=\#\mathcal{O}^\times/\mathbb{Z}^\times$, by Table \ref{Table_xi} we know that $\ell([e])\in\{1,2,3,4,6,12\}$. We need to separate the cases:
	\begin{enumerate}[(i)]
		\item $(D,N)=(2,9),(2,11),(3,4),(5,2)$ and $(13,1)$. In these cases we have that $\#\mathcal{O}^\times/\mathbb{Z}^\times=1$, so $\Gamma_{p}\backslash\mathcal{T}_{p}$ is the stable reduction-graph $\Gamma_{p}(\xi)\backslash\mathcal{T}_{p}$ with $(p+1)/2$ edges of length $1$.
		\item $D=2$, $N=1$. In this case we have
		$\mathcal{O}^\times/\mathbb{Z}^\times=\left\{1,i,j,k,\frac{1\pm i\pm j\pm k}{2}\right\}$
		and there are transformations of order 2 and of order 3. The order 2 transformations are
		$\Phi_p(i)=\left(\begin{smallmatrix} \sqrt{-1} & 0 \\ 0 & \sqrt{-1}\end{smallmatrix}\right),  \Phi_p(j)=\left(\begin{smallmatrix} 0 & 1 \\ -1 & 0 \end{smallmatrix}\right)$ and $\Phi_p(k)=\left(\begin{smallmatrix} 0 & \sqrt{-1} \\ \sqrt{-1} & 0 \end{smallmatrix}\right)$,
		and the order 3 transformations are
		$\Phi_p\left(\frac{1\pm i+j+k}{2}\right)=\frac{1}{2}\left(\begin{smallmatrix}1\pm \sqrt{-1} & 1+\sqrt{-1} \\ -1+\sqrt{-1} & 1\mp \sqrt{-1} \end{smallmatrix}\right),$
		$\Phi_p\left(\frac{1+ i\pm j+k}{2}\right)=\frac{1}{2}\left(\begin{smallmatrix}1+ \sqrt{-1} & \pm 1+\sqrt{-1} \\ \mp1+\sqrt{-1} & 1- \sqrt{-1} \end{smallmatrix}\right),$
		$\Phi_p\left(\frac{1+ i+ j\pm k}{2}\right)=\frac{1}{2}\left(\begin{smallmatrix}1+ \sqrt{-1} & 1\pm \sqrt{-1} \\ -1\mp\sqrt{-1} & 1- \sqrt{-1} \end{smallmatrix}\right),$	
		$\Phi_p\left(\frac{1\pm i- j\mp k}{2}\right)= \frac{1}{2}\left(\begin{smallmatrix}1\pm \sqrt{-1} & -1 \mp\sqrt{-1} \\ 1\mp\sqrt{-1} & 1\mp \sqrt{-1} \end{smallmatrix}\right).$
		
		Studying the fixed points of each transformation in $\mathbb{P}^1(\mathbb{F}_p)$ one can see that, under the restriction that $\left(\frac{-1}{p}\right)=1$ (which is given by the chosen presentation for the algebra), each order 2 transformation always has two fixed points in $\mathbb{P}^1(\mathbb{F}_p)$, and that an order 3 transformation has two fixed points in $\mathbb{P}^1(\mathbb{F}_p)$ if $p\equiv 1$ mod 3 and zero fixed points if $p\equiv 2$ mod 3. One also checks that there is no point fixed at the same time by an order 2 transformation and an order 3 transformation. From this we conclude that $c_3=2$ if $p\equiv1$ mod $3$ and $c_3=0$ if $p\equiv2$ mod 3. Next, one can see that the three unoriented edges of $\Gamma_p(2)\backslash\mathcal{T}_p$ corresponding to the six fixed points of order 2 transformations are all identified by the transformations of order 3. Thus we are left with only one edge $[y]$ of length 2. Since $c_2$ counts the number of oriented edges of order 2, now we could have either $c_2=2$ (if $[e]\neq[\overline{e}]$) or $c_2=1$ (if $[e]=[\overline{e}]$). An easy computation shows that we are in the letter case. Thus $c_2=1$. Finally, by Lemma \ref{formula}, we obtain $p+1=12c_1+6c_2+4c_3$. So we have that if $c_3=2$ then $c_1=(p-13)/12$ and if $c_3=0$ then $c_1=(p-5)/12$.
		
		\item $D=2$, $N=3$. Here, since $\#\mathcal{O}^\times/\mathbb{Z}^\times=3$, we have $c_2=0$. Lemma \ref{formula} gives $p+1=3c_1+c_3$. After studying the order $3$ transformations we see that $c_3=2$ if $p\equiv1$ mod $3$ and $c_3=0$ if $p\equiv2$ mod 3. Thus $c_1=(p-1)/3$ if $p\equiv1$ mod 3 and $c_3=(p+1)/3$ if $p\equiv2$ mod 3.
		
		\item $D=3$, $N=1$. Here $\#\mathcal{O}^\times/\mathbb{Z}^\times=6$. One sees that the points fixed by order 2 transformations are never fixed by order 3 transformations. Lemma \ref{formula} gives $p+1=6c_1+3c_2+2c_1$. The order 3 transformations give $c_3=1$ when $p\equiv 1$ mod 3 and $c_3=0$ when $p\equiv2$ mod 3, and the order 2 transformations give some unoriented edges that the order 3 transformations take to the same class $[e]\in\mathrm{Ed}(\Gamma_p(2)\backslash\mathcal{T}_p)$. Since $[e]\neq[\overline{e}]$, we have $c_2=2$.

		\item $D=3$, $N=2$. Here $\#\mathcal{O}^\times/\mathbb{Z}^\times=2$, so $c_3=0$. The order 2 transformation gives two oriented edges in different classes, so $c_2=2$. Lemma \ref{formula} gives $p+1=2c_1+c_2$, so $c_1=\frac{p-1}{2}$.
		
		\item $D=5$, $N=1$. In this case $\#\mathcal{O}^\times/\mathbb{Z}^\times=3$, so $c_2=0$. Similarly we have $c_3=2$ and $c_1=(p-1)/3$ when $p\equiv 1$ mod 3 and $c_3=0$ and $c_1=(p+1)/3$ when $p\equiv 2$ mod 3.
	\end{enumerate}
\end{proof}

\begin{remark}
	We already explained how that are restricting ourselves to some infinite set of primes $p$ such that $\left(\frac{a}{p}\right)=1$, with $a=-1$ when $D=2,3$, or $a=-2$ when $D=5, 13$. 
	This means that $p\equiv 1\,\mathrm{mod}\,4$ when $D=2,3$ and $p\equiv 1,3\,\mathrm{mod}\,8$ when $D=5,13$.
	Thus, after section \ref{null-trace_condition} we find out that, at least in the cases $(D,N)=(2,1),(3,1),$ the condition $p\equiv 1\,\mathrm{mod}\,4$ is not restrictive, since the primes satisfying it also satisfy null-trace condition of Theorem \ref{Theorem_Mumford_curve}.
\end{remark}

\subsection{Genus formulas}
	After Corollary \ref{Corollary} we have 
a covering $\Gamma_{p}(\xi)\backslash\mathcal{H}_{p}\twoheadrightarrow\Gamma_{p}\backslash\mathcal{H}_{p}$ of degree $\mathcal{O}^{\times}/\mathbb{Z}^{\times}$, where the quotient $\Gamma_{p}(\xi)\backslash\mathcal{H}_{p}$ is algebraisable and has genus $(p+1)/2$. Therefore, following \cite{Gerritzen_vanderPut1980} we can apply the Riemann-Hurwitz formula to compute the genus of the algebraic curve $\Gamma_{p}\backslash\mathcal{H}_{p}$:
\[
p-1=\#(\mathcal{O}^{\times}/\mathbb{Z}^{\times})(2g-2)+\sum_{d\mid \#(\mathcal{O}^{\times}/\mathbb{Z}^{\times}) } w_{d},
\]
where $w_{d}$ is the number of points on $\Gamma_{p}\backslash\mathcal{H}_{p}$ fixed by some transformation in $\Gamma_{p}$ of order $d > 1$.

\begin{prop}\label{genus_formulas}
Under the assumptions of Theorem $\ref{Theorem_Mumford_curve}$, the genus of $\Gamma_{p}\backslash\mathcal{T}_{p}$ is given in Table $\ref{Table_genus}$,
where
\begin{enumerate}[$(i)$]
\item
$\delta_{p}(2,1):=w_{2,1}+w_{2,2}+w_{2,3}$, with \\
$
w_{2,i}:=\frac{1}{2}\#\{\alpha=(a_{0},a_{1},a_{2},a_{3})_{H}\in\mathcal{O}\mid\, \mathrm{Nm}(\alpha)=p,\,\alpha\equiv 1\,\mathrm{mod}\ 2\mathcal{O},\;a_{i}=0\},\;i=1,2,3.
$
\item
$\delta_p(3,1):=w_{2,1}+w_{2,2}+w_{2,3}$, with \\
$
w_{2,1}:=\frac{1}{2}\#\{\alpha=(a_{0},a_{1},a_{2},a_{3})_{H}\in\mathcal{O}\mid\, \mathrm{Nm}(\alpha)=p,\alpha\equiv 1\,\mathrm{mod}\ 2\mathcal{O},\;a_{1}=0\},
$ \\
$w_{2,2}:=\frac{1}{2}\#\{\alpha=(a_{0},a_{1},a_{2},a_{3})_{H}\in\mathcal{O}\mid\, \mathrm{Nm}(\alpha)=p,\alpha\equiv 1\,\mathrm{mod}\ 2\mathcal{O},\;a_{1}+3a_{2}=0\},
$ \\
$
w_{2,3}:=\frac{1}{2}\#\{\alpha=(a_{0},a_{1},a_{2},a_{3})_{H}\in\mathcal{O}\mid\, \mathrm{Nm}(\alpha)=p,\alpha\equiv 1\,\mathrm{mod}\ 2\mathcal{O},\;a_{1}-3a_{2}=0\}.
$
\item
$\delta_p(3,2):=\dfrac{1}{2}\#\{\alpha=(a_{0},a_{1},a_{2},a_{3})_{H}\in\mathcal{O}\mid \mathrm{Nm}(\alpha)=p,\,\alpha\equiv 1\,\mathrm{mod}\ \xi\mathcal{O},\;a_{1}-3a_{2}=0\}.$
\end{enumerate}
\end{prop}

\begin{table}[h]
\vspace{-1em}
\captionsetup{font=footnotesize}
\centering
\scalebox{0.75}{
\begin{tabular}{|c|c|c||c|c|c|} \hline
	$D$ & $N$ & genus & $D$ & $N$ & genus\\ 
	\hline \hline
	2 & 1 & $\vphantom{\int\bigcup^n\frac{\frac{i}{2}}{6}}$ $\dfrac{1}{24}\left(p+23-\delta_{p}(2,1)-8\left(1-\left(\frac{3}{p}\right)\right)\right)$ & 3 & 2 & $\dfrac{1}{4}\left(p+3-\delta_{p}(3,2)\right)$ \\
	2 & 3 & $\dfrac{1}{6}\left(p+5-2\left(1-\left(\frac{-3}{p}\right)\right)\right)$ & 3 & 4 & $(p+1)/2$ \\
	2 & 9 & $(p+1)/2$ & 5 & 1 & $\dfrac{1}{6}\left(p+5-2\left(1-\left(\frac{-3}{p}\right)\right)\right)$ \\
	2 & 11 & $(p+1)/2$ & 5 & 2 & $(p+1)/2$ \\
	3 & 1 & $\vphantom{\int\bigcup^n\frac{\frac{i}{2}}{6}}$ $\dfrac{1}{12}\left(p+11-\delta_{p}(3,1)-2\left(1-\left(\frac{3}{p}\right)\right)\right)$  & 13 & 1 & $(p+1)/2$ \\ \hline
\end{tabular}
}
\caption{Genus of $\Gamma_{p}\backslash\mathcal{T}_{p}$}\label{Table_genus}
\vspace{-1.5em}
\end{table}

\begin{proof}
For the case $D=2$, $N=1$ we refer the reader to \cite[Ch. IX]{Gerritzen_vanderPut1980}. Since all the cases are similarl, we will do just the case $D=3$, $N=1$, where the group $\Phi_{p}(\mathcal{O}^{\times})/\mathbb{Z}^{\times}$ is the biggest among the ones we are considering.
For each order $2$ element $u_{2,i}\in\Phi_{p}(\mathcal{O}^{\times})/\mathbb{Z}^{\times}$, we compute
$
w_{2,i}:=\#\{[z]\in\Gamma_{p}\backslash\mathcal{T}_{p}\mid\, u_{2,i}z=\gamma z,\; \gamma\in\Gamma_{p}(2)\}.
$
Let us study each unit of order 2.
\begin{itemize}
\item
$u_{2,1}=\left(\begin{smallmatrix}\sqrt{-1} & 0 \\ 0 & \sqrt{-1} \end{smallmatrix}\right)$. The equation $uz=\gamma z$, with $\gamma=\left(\begin{smallmatrix}A & B \\ C & D\end{smallmatrix}\right)\in\Gamma_{p}(2)$ gives 
$$
z=-\frac{A+D}{2C}\pm\frac{\sqrt{(\frac{A+D}{2})^{2}-BC}}{C},
$$ 
so $(\frac{A+D}{2})^{2}-BC=a_{0}+3(a_{2}^{2}+a_{3}^{2})=p-a_{1}^{2}$, which does not belong to $\mathbb{Q}_{p}^{\times 2}$ if and only if $a_{1}=0$. Therefore we obtain that $w_{2,1}$ is equal to $2$ times the number of generators $\alpha$ of $\Gamma_{p}(\xi)$ such that $(a_{0},a_{1},a_{2},a_{3})_{H}$ in the basis of $H=\left(\frac{-1,-3}{\mathbb{Q}}\right)$ has $a_{1}=0$, and this is equal to
$
1/4\cdot 2\cdot \#\{\alpha=(a_{0},a_{1},a_{2},a_{3})_{H}\in\mathcal{O}\mid\, \mathrm{Nm}(\alpha)=p,\alpha\equiv 1\,\mathrm{mod}\,2\mathcal{O},\;a_{1}=0\},
$
since we have to exclude, for each $\alpha$ in this set, the elements $\overline{\alpha}$, $-\alpha$ and $-\overline{\alpha}$.
\item
$u_{2,2}=\left(\begin{smallmatrix}1/2\sqrt{-1} &  1/2\\ -3/2 & -1/2\sqrt{-1} \end{smallmatrix}\right)$. Analogous arguments show that the number $w_{2,2}$ is equal to $2$ times the number of generators $\alpha=(a_{0},a_{1},a_{2},a_{3})_{H}$ of $\Gamma_{p}(\xi)$ with $a_{1}+3a_{2}=0$.
\item
$u_{2,3}=\left(\begin{smallmatrix}1/2\sqrt{-1} & -1/2\\ 3/2& -1/2\sqrt{-1} \end{smallmatrix}\right)$. Similar computations in this case give that $w_{2,3}$ is equal to $2$ times the number of generators $\alpha=(a_{0},a_{1},a_{2},a_{3})_{H}$ of $\Gamma_{p}(\xi)$ with $a_{1}-3a_{2}=0$. 
\end{itemize}
For the order $3$ transformations $u_{3,i}\in\Phi_{p}(\mathcal{O}^{\times})/\mathbb{Z}^{\times}$, $i=1,2$, we have to compute 
$w_{3,i}:=\{[z]\in\Gamma_{p}\backslash\mathcal{T}_{p}\mid\, uz= z\}$. 
This is $1-\left(\frac{3}{p}\right)$ in each case, and the formula is then completed.
\end{proof}

\begin{cor}
Under the assumptions of Theorem $\ref{Theorem_Mumford_curve}$, $\Gamma_{p,+}\backslash\mathcal{T}_{p}$ is described as follows:
\begin{enumerate}[$(a)$]
\item 
It has $c_{1}+c_{2}+c_{3}$ unoriented edges joining two vertices, with $c_{1},c_{2},c_{3}$ as in Table $\ref{Table_edges_graph}$.
\item
Its genus is $g_{+}=c_{1}+c_{2}+c_{3}-1.$
\item
The number of edges with lengths is obtained by multiplying by $2$ the formulas of Table $\ref{Table_edges_graph}$.
\end{enumerate}
\end{cor}

\begin{proof}
The curve $\Gamma_{p,+}\backslash\mathcal{H}_{p}$ is a degree $2$ covering of $\Gamma_{p}\backslash\mathcal{H}_{p}$, so it is immediate to see that 
\[
g_{+}=\#\{e\in\mathrm{Ed}(\Gamma_{p}\backslash\mathcal{T}_{p})\mid\, -e\neq e\}-1+\#\{e\in\mathrm{Ed}(\Gamma_{p}\backslash\mathcal{T}_{p})\mid\, -e=e\}=c_{1}+c_{2}+c_{3}-1.
\]
The rest of the statement is then also clear.
\end{proof}

\begin{remark}
	The formulas in Proposition \ref{genus_formulas} generalise those computed in \cite{Gerritzen_vanderPut1980} relative to the family of Shimura curves $X(2p,1)$, with $p\equiv 1\;\mathrm{mod}\ 4$. 
	We refer the reader to \cite{vanderPut1992}, where the author explains the explicit relations of these computations with the family $X(2p,1)$. 
\end{remark}

\begin{remark}
After Lemma \ref{lemma_D3} we can rewrite $\delta_{p}(3,1)$ as $\delta_{p}(3,1):=w_{2,1}+w_{2,2}+w_{2,3}$, where
\begin{itemize}
\item[-] $w_{2,1}:=\frac{1}{2}\#\{(a_{0},a_{2},a_{3})\in\mathbb{Z}^{3}\mid \, a_{0}^{2}+3(a_{2}^{2}+a_{3}^{2})=p, \ a_{0}+a_{3}\equiv 1\;\mathrm{mod}\ 2,\;a_{2}\equiv 0\;\mathrm{mod}\ 2\},$

\item[-] $w_{2,2}:=\frac{1}{2}\#\{(a_{0},a_{1},a_{2},a_{3})\in\mathbb{Z}^{4}\mid \, a_{0}^{2}+a_{1}^{2}+3(a_{2}^{2}+a_{3}^{2})=p,\ a_{0}+a_{3}\equiv\,1\;\mathrm{mod}\ 2,\,a_{1}+a_{2}\equiv 0\ \mathrm{mod}\ 2,\,a_{1}+3a_{2}=0\},$

\item[-] $w_{2,3}:=\frac{1}{2}\#\{(a_{0},a_{1},a_{2},a_{3})\in\mathbb{Z}^{4}\mid \, a_{0}^{2}+a_{1}^{2}+3(a_{2}^{2}+a_{3}^{2})=p,\ a_{0}+a_{3}\equiv\,1\ \mathrm{mod}\ 2,\,a_{1}+a_{2}\equiv 0\;\mathrm{mod}\ 2,\,a_{1}-3a_{2}=0\}$.
\end{itemize}
The same works for $(D,N)=(2,1)$ (\cite[Ch. IX]{Gerritzen_vanderPut1980}). 
Here $\delta_{p}(2,1)=w_{2,1}+w_{2,2}+w_{2,3}$, with 
\begin{align*}
w_{2,1}:=1/2\cdot\#\{(a_{0},a_{2},a_{3})\in\mathbb{Z}^{3}\mid\, a_{0}^{2}+a_{2}^{2}+a_{3}^{2}=p,
\ a_{0}\equiv\,1;\mathrm{mod}\,2,\;a_{2}\equiv a_{3}\equiv 0\;\mathrm{mod}\,2\}, \\
w_{2,2}:=1/2\cdot\#\{(a_{0},a_{1},a_{3})\in\mathbb{Z}^{3}\mid\, a_{0}^{2}+a_{1}^{2}+a_{3}^{2}=p,\ a_{0}\equiv\,1;\mathrm{mod}\,2,\,a_{1}\equiv a_{3}\equiv 0\;\mathrm{mod}\,2\}, \\
w_{2,3}:=1/2\cdot\#\{(a_{0},a_{1},a_{2})\in\mathbb{Z}^{3}\mid\, a_{0}^{2}+a_{1}^{2}+a_{2}^{2}=p,\ 
a_{0}\equiv\,1;\mathrm{mod}\,2,\,a_{1}\equiv a_{2}\equiv 0\;\mathrm{mod}\,2\}.
\end{align*}
\end{remark}

\begin{remark}
The genus $g_{+}$ is actually the genus of the special fibre of the Drinfel'd model $\mathcal{X}(Dp,N)$ of the considered Shimura curves. Indeed, the reader can check that this coincides with the usual genus formula for the Shimura curve $X(Dp,N)$
(cf. \cite{ShimuraBook1970}).
\end{remark}

\section{Examples: computation of reduction-graphs}\label{Sec5}

Here we describe an algorithm, implemented in Magma, to obtain explicit examples illustrating our method\footnote{The Magma code described can be found on the second author's webpage \texttt{piermarcomilione.wordpress.com}.}. We will consider the family of Shimura curves with $D_B=3p$ and $N=2$. 

First we compute the definite quaternion algebra $H=\left(\frac{-1,-3}{\mathbb{Q}}\right)$, a basis for an Eichler order $\mathcal{O}\subseteq H$ of level $N=2$, and an element $\xi\in\mathcal{O}$ satisfying the right-unit property in $\mathcal{O}$.
\begin{Verbatim}[fontsize=\footnotesize]
> D := 3; N:=2; H, O := Data(D,N); NmH, NmO := Normic_form(H,O);
> xi := choose_xi(O); xi;
-1/2 - 1/2*I - 1/2*J + 1/2*K
\end{Verbatim}
We fix $p=13$, a prime that satisfies the null-trace condition with respect to $\xi$. By Theorem \ref{Theorem_Mumford_curve} we know that $\Gamma_{13}(\xi)$ is a Schottky group and the stable reduction-graph of the associated Mumford curve is a graph obtained via pair-wise identification of the boundary edges of a graph consisting of one vertex and $p+1=14$ edges (see figure \ref{graph1}).
\begin{figure}[h]
\captionsetup{font=footnotesize}
\begin{center}
\includegraphics[scale=0.2]{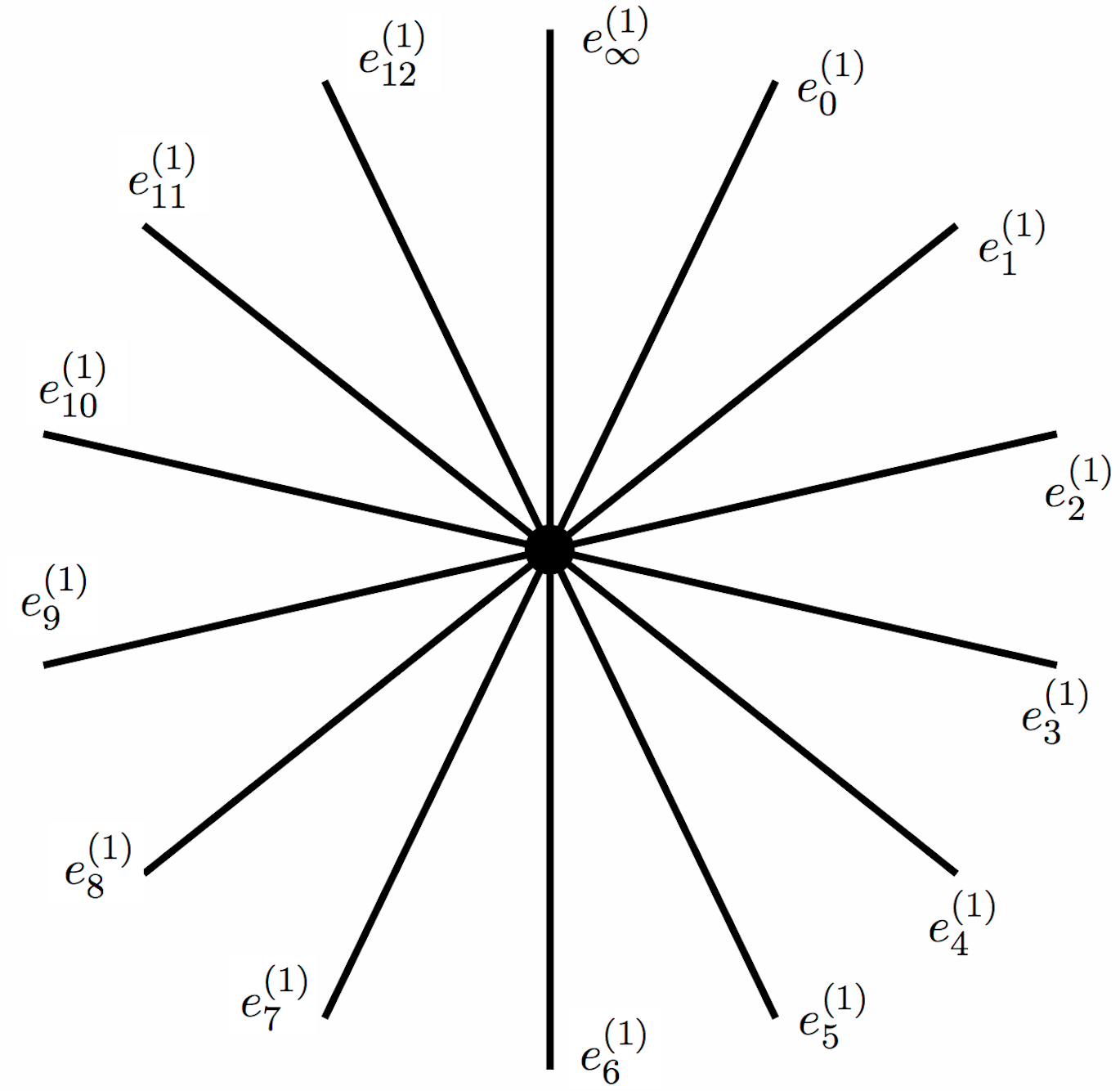}
\caption[graph1]{Reduction of the fundamental domain $\mathcal{F}_{13}(-\frac{1}{2}-\frac{1}{2}i-\frac{1}{2}j+\frac{1}{2}ij)$}\label{graph1}
\vspace{-1.5em}
\end{center}
\end{figure}

In order to compute these identifications we need to compute a system of generators for $\Gamma_p(\xi)$. These are first computed as quaternions in $H$ and then as $p$-adic matrices in $\mathrm{GL}_{2}(\mathbb{Q}_{p})$. Note that there are $(p+1)/2=7$ of them, as predicted by Proposition \ref{rank}. The symbol \verb+i+ appearing in the matrix description below is the $13$-adic number $\sqrt{-1}$.
\begin{Verbatim}[fontsize=\footnotesize]
> p := 13;
> gens_xi, gens_Schottky := generators_Gammas(H,O,p,xi);
> gens_xi;
[ -3 + 1/2*I - 1/2*J + K, -3 - 1/2*I + 1/2*J + K, -3 + I - J,
-1 + 2*K, -1 + 3/2*I - 3/2*J + K, -1 - 3/2*I + 3/2*J + K, -1 - 3*I - J ]
> gens_Schottky;
[   [ 1/2*(i - 6) 1/2*(2*i - 1)]	 [1/2*(-i - 6) 1/2*(2*i + 1)]
    [1/2*(6*i + 3)  1/2*(-i - 6)],       [1/2*(6*i - 3)   1/2*(i - 6)],

    [ 1/2*(3*i - 2)  1/2*(2*i - 3)]     [1/2*(-3*i - 2)  1/2*(2*i + 3)]
    [ 1/2*(6*i + 9) 1/2*(-3*i - 2)],    [ 1/2*(6*i - 9)  1/2*(3*i - 2)],

    [ i - 3     -1]      [-3*i - 1       -1]      [ -1 2*i]
    [     3 -i - 3],     [       3  3*i - 1],     [6*i  -1]  ]
\end{Verbatim}
We can now compute the fundamental domain $\mathcal{F}_{p}(\xi)\subseteq \mathcal{H}_{p}$ for the action of $\Gamma_{p}(\xi)$. The function \verb+fundamental_domain+ gives as output the radii $1/\sqrt{p}$ of the balls constituting the boundary of $\mathcal{F}_{p}(\xi)$, as well as the pair-wise identification of the boundary of these balls (cf. Theorem \ref{Theorem_Mumford_curve}). This is done by reducing modulo $p$ the fixed points of $\gamma,\gamma^{-1}$ for all the generators $\gamma$ previously computed. Each ball is identified by the coordinates of its centre, which is a point in $\mathbb{P}^{1}(\mathbb{Q}_{p})$.
\begin{Verbatim}[fontsize=\footnotesize]
> radius, pairing := fundamental_domain(H,O,p,xi,gens_xi);
> pairing;
[  <( 9  1), (11  1)>,  <( 1  1), ( 2  1)>,  <( 5  1), ( 7  1)>,  <( 3  1), (10  1)>, 
 <( 8  1), ( 1  0)>,  <( 0  1), ( 6  1)>,  <( 4  1), (12  1)>  ]
\end{Verbatim}
In figure \ref{dominio_padico} we represent this $p$-adic fundamental domain, where one has to read the labels of the balls as a notation for the pair-wise identifications given by \verb+pairing+: the interior (resp. exterior) of the ball $X$ is identified with the exterior (resp. interior) of the ball $X^{-1}$. 

\begin{figure}[h]
\captionsetup{font=footnotesize}
\begin{center}
\includegraphics[scale=0.2]{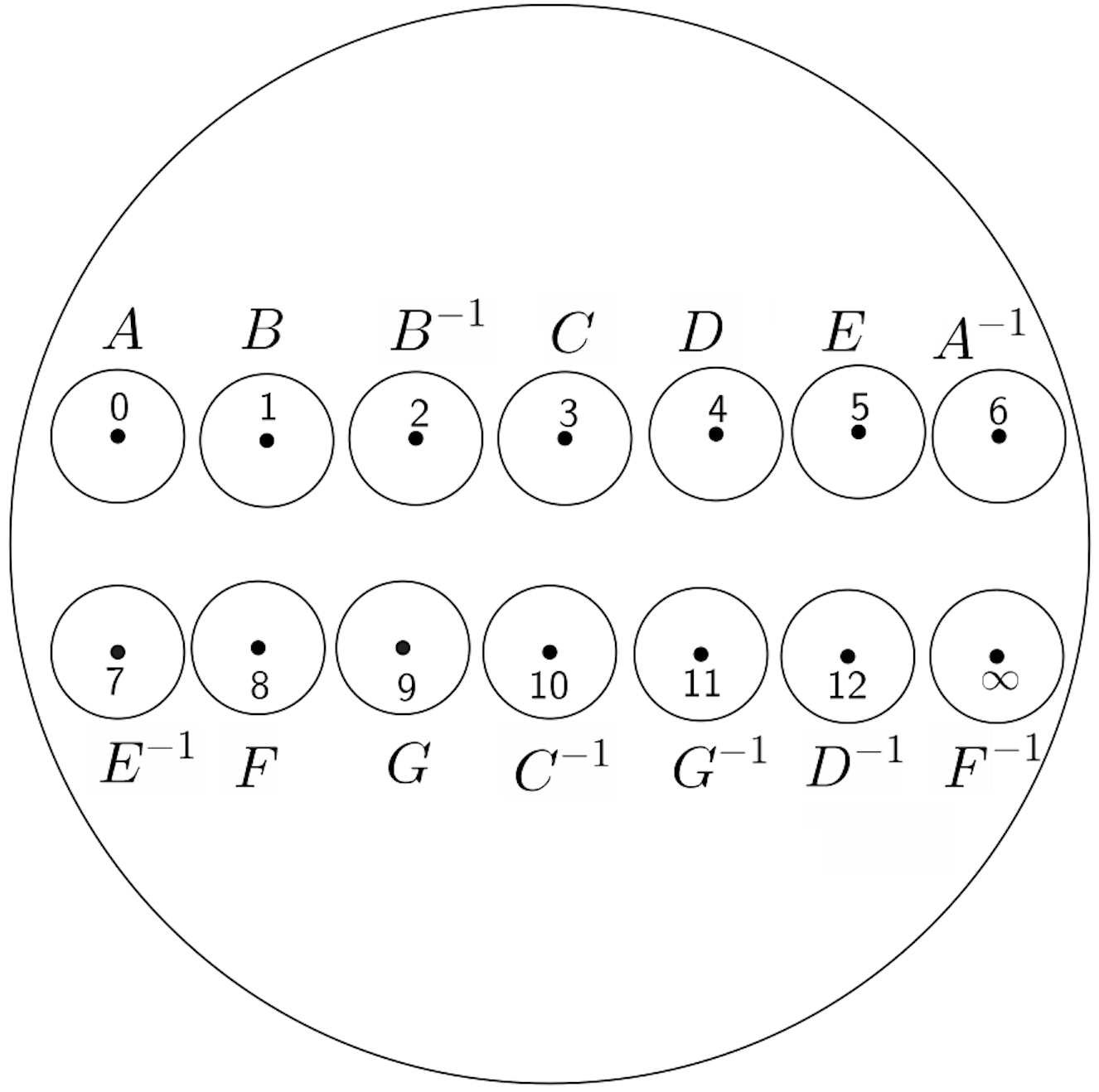}
\vspace{-1.5em}
\end{center}
\caption{Fundamental domain $\mathcal{F}_p(\xi)$ for the action of $\Gamma_p(\xi)$ in $\mathcal{H}_p$}\label{dominio_padico}
\vspace{-0.5em}
\end{figure}

In figure \ref{graph2}, we can see the reduction-graph of the Mumford curve associate to $\Gamma_{p}(\xi)$, which is obtained by reduction of the rigid analytic variety $\Gamma_{p}(\xi)\backslash\mathcal{F}_{p}(\xi)$.

\begin{figure}[h]
\captionsetup{font=footnotesize}
\begin{center}
\includegraphics[scale=0.2]{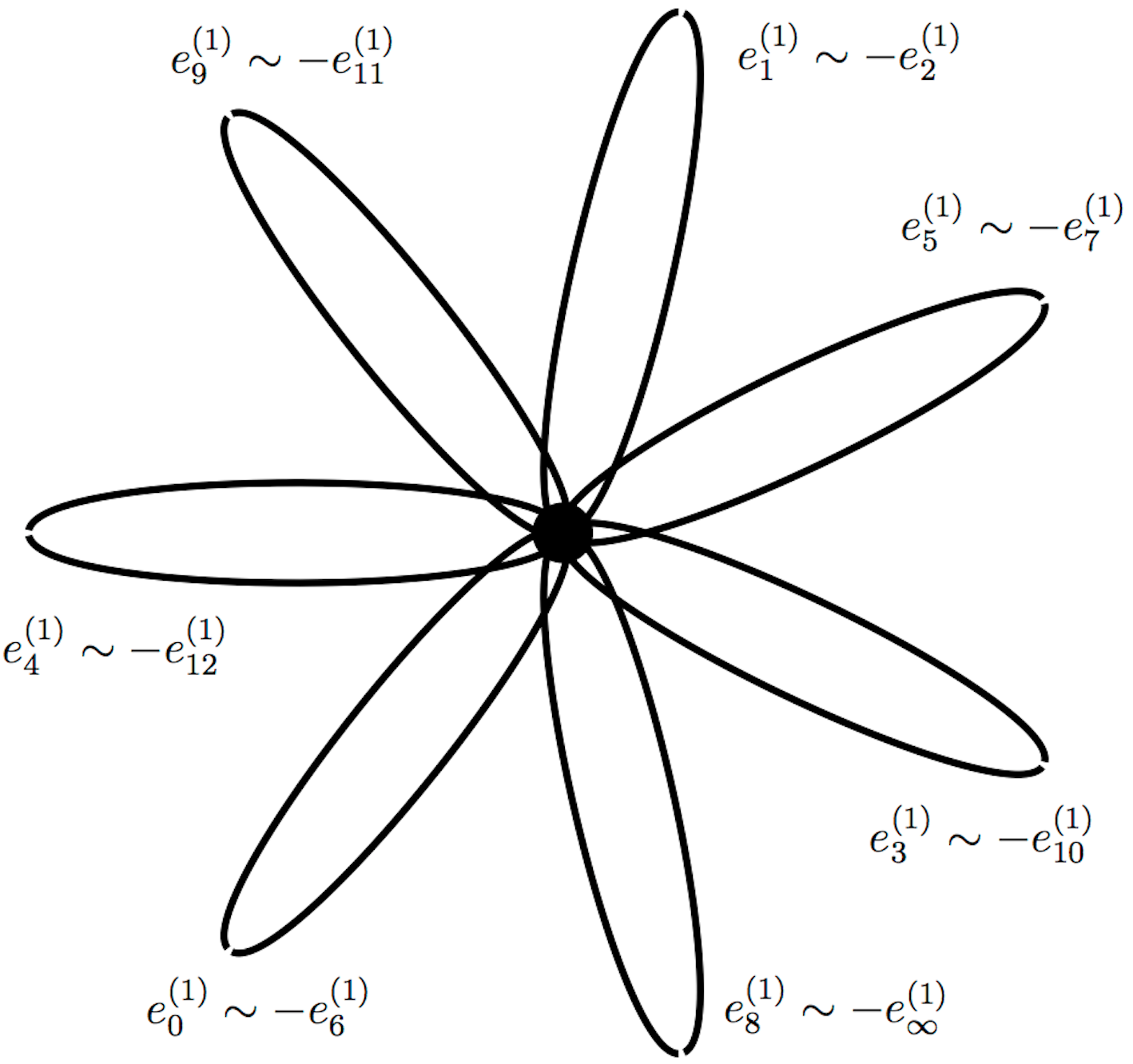}
\caption{Stable reduction-graph of the Mumford curve associated with $\Gamma_{13}(-\frac{1}{2}-\frac{1}{2}i-\frac{1}{2}j+\frac{1}{2}ij)$}\label{graph2}

\end{center}
\vspace{-1.5em}
\end{figure}

Now we compute $\Gamma_{p}\backslash\mathcal{T}_{p}$. 
Although one can easily do this with the formulas of tables \ref{Table_edges_graph} and \ref{Table_genus}, here we use an alternative algorithm that we designed to obtain the desired graph. We do it by letting $\Gamma_{p}(\xi)$ act on $(\mathcal{O}^{\times}/\mathbb{Z}^{\times})\backslash\mathcal{T}_{p}=(\Gamma_{p}/\Gamma_{p}(\xi))\backslash\mathcal{T}_{p}$.
We first compute the quotient of the reduction of $\mathcal{F}_{p}(\xi)$ (see Figure \ref{graph1}) by the action of $\mathcal{O}^{\times}/\mathbb{Z}^{\times}$. This is done using \verb+UnitsAction+. 
\begin{Verbatim}[fontsize=\footnotesize]
classes := UnitsAction(H,O,p); classes;
[  {( 8  1), ( 0  1)},  {( 1  1), (11  1)},  {( 2  1), ( 9  1)}, {(10  1), ( 3  1)}, 
{(12  1), ( 4  1)},  {( 5  1)}, {( 7  1)}, {( 6  1), ( 1  0)}  ]
\end{Verbatim}
For example, we see that $e_{8}^{(1)}$ and $e_{0}^{(1)}$ are identified, and that $e_{5}^{(1)}$ remains alone in its class. 

Finally, with \verb+DescriptionGraph+ we compute $\Gamma_{13}\backslash\mathcal{T}_{13}$, which corresponds to the special fibre at $13$ of the Atkin-Lehner quotient of $X(3\cdot 13,2)$.

\begin{Verbatim}[fontsize=\footnotesize]
aller_retour, loops := DescriptionGraph(H,O,p,xi,gens_xi);
> aller_retour;
[  <{ (10  1),( 3  1) }, 1>,  <{ (12  1), ( 4  1) }, 1>  ]
> loops;
[  <{ ( 8  1), ( 6  1), ( 0  1), ( 1  0) }, 1>,  
<{ ( 2  1), ( 1  1), (11  1), ( 9  1) }, 1>,  <{ ( 7  1), ( 5  1) }, 2>  ]
\end{Verbatim}
We can see that this graph has two \textit{aller-retour} edges (i.e., edges $e$ with $-e=e$) of length $1$ and one \textit{loop} of length $2$ (obtained from the identification of $e_5^{(1)}$ and $-e_{7}^{(1)}$). In figure \ref{graph5} we represent the reduction-graphs $\Gamma_{13}\backslash\mathcal{T}_{13}$ and $\Gamma_{13,+}\backslash\mathcal{T}_{13}$ associated with $X(3p,2)$ at the prime $p=13$ .
\begin{figure}[h!]
\captionsetup{font=footnotesize}
\begin{center}
	\includegraphics[scale=0.15]{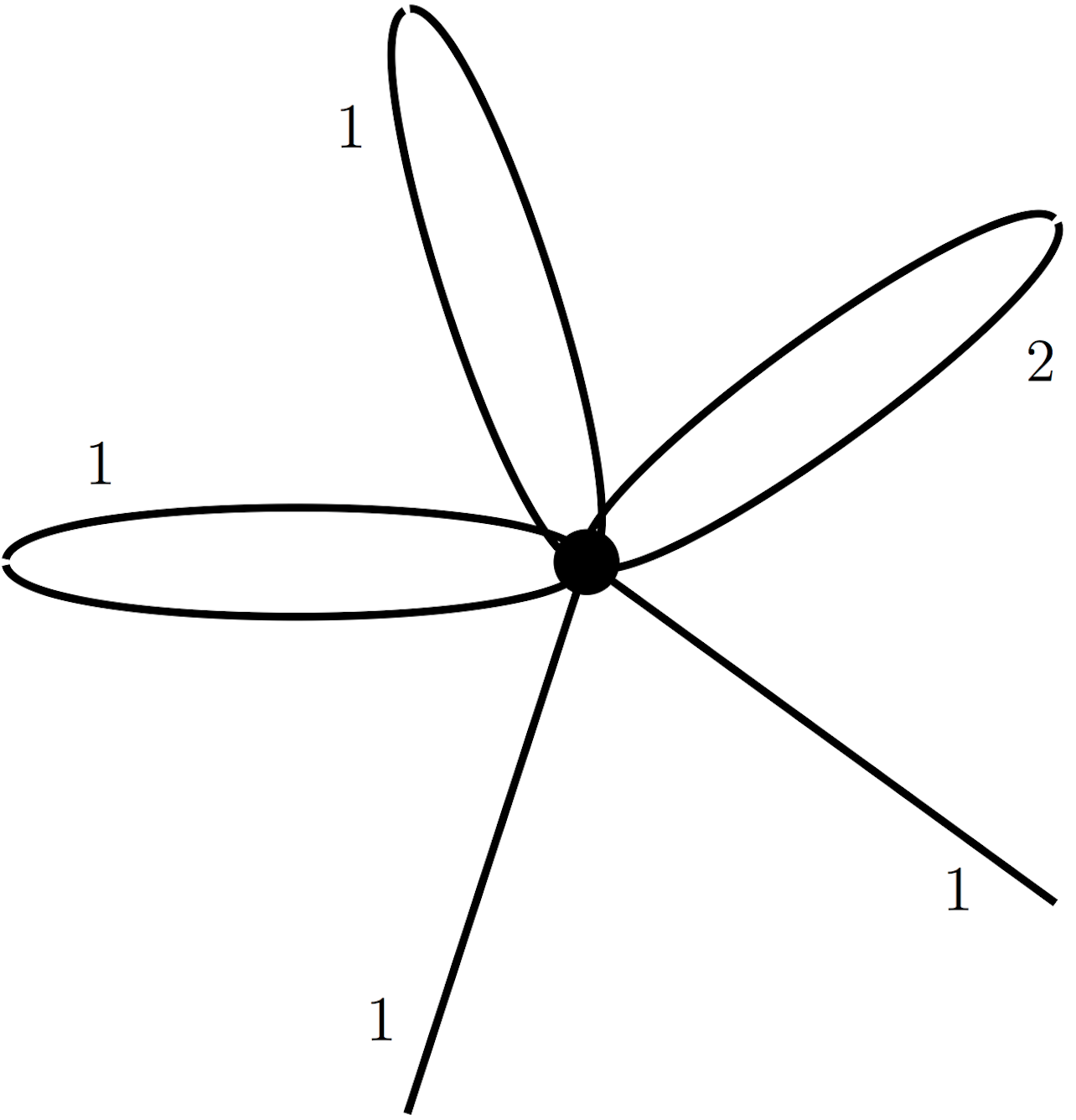}\quad\quad
	\includegraphics[scale=0.2]{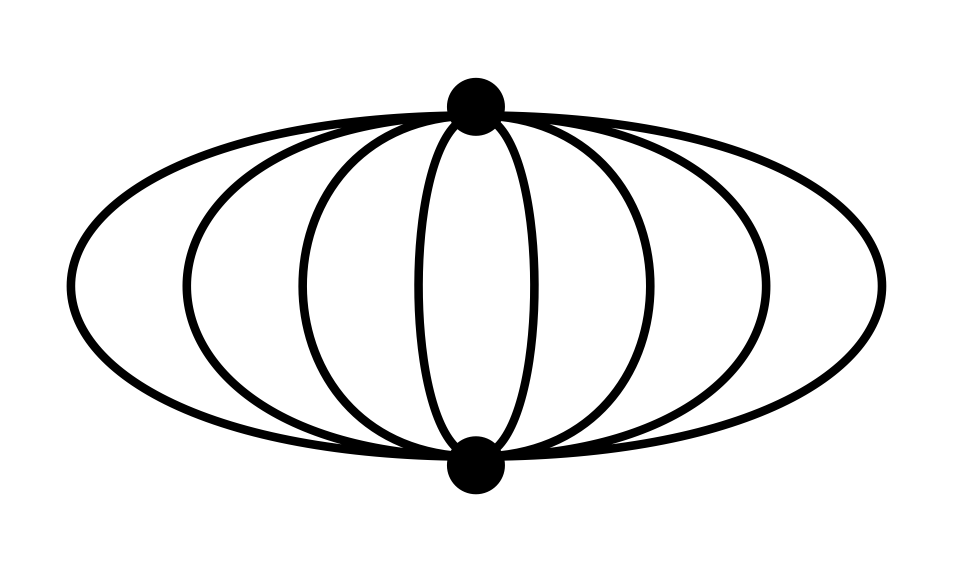}
	\vspace{-1.5em}
\end{center}
\caption{Reduction-graphs with lengths $\Gamma_{13}\backslash\mathcal{T}_{13}$ and $\Gamma_{13,+}\backslash\mathcal{T}_{13}$ for $X(3\cdot 13, 2)$}\label{graph5}
\vspace{-0.5em}
\end{figure}

To conclude, in figure \ref{ex2_2} we show the final reduction-graphs in another example: the reduction-graphs $\Gamma_{p}\backslash\mathcal{T}_{p}$ and $\Gamma_{p,+}\backslash\mathcal{T}_{p}$ when $D=3,\,N=1$ and $p=61$.

\begin{figure}[h]
\captionsetup{font=footnotesize}
\begin{center}
	\includegraphics[scale=0.16]{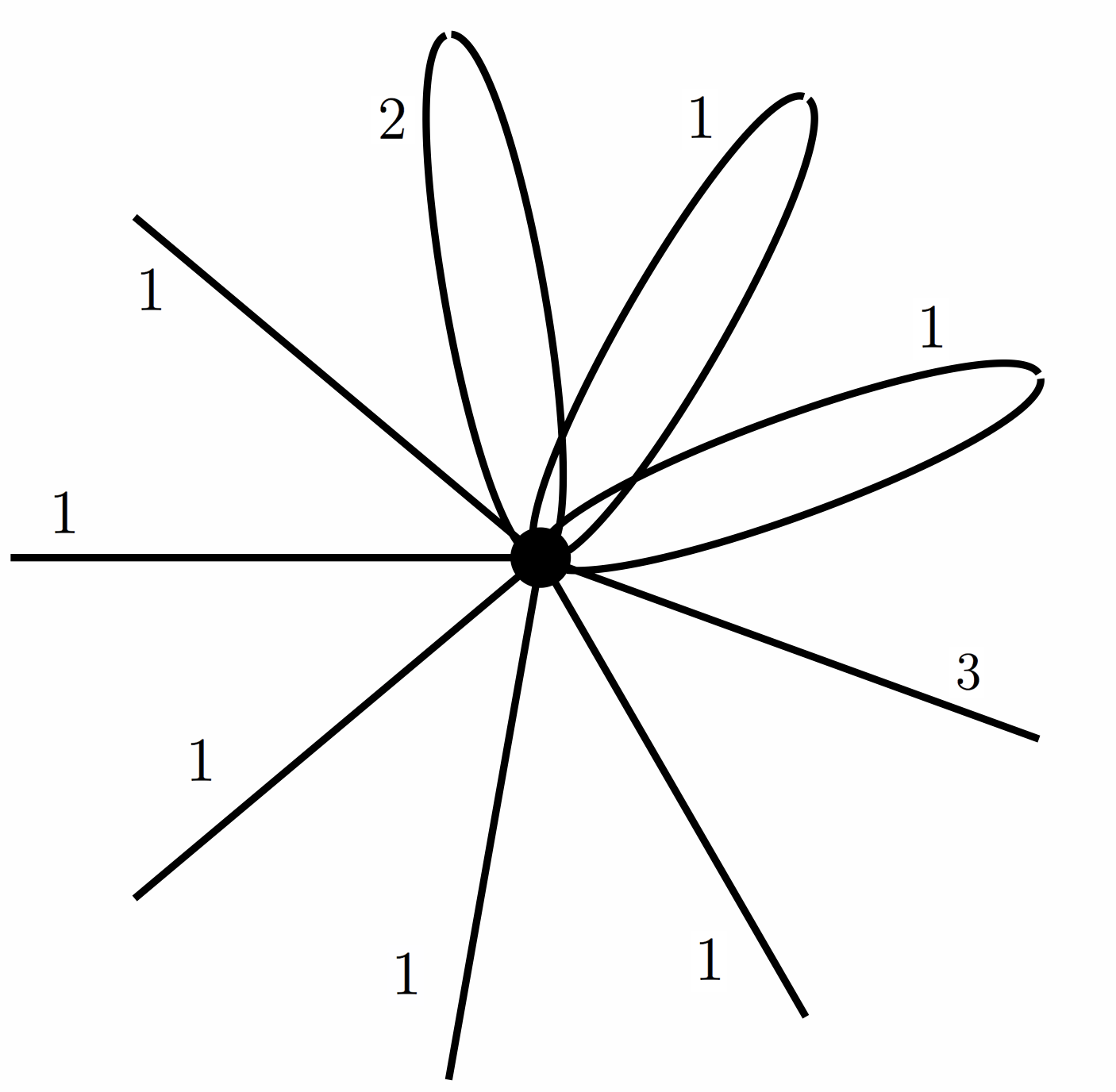}\quad\quad
	\includegraphics[scale=0.18]{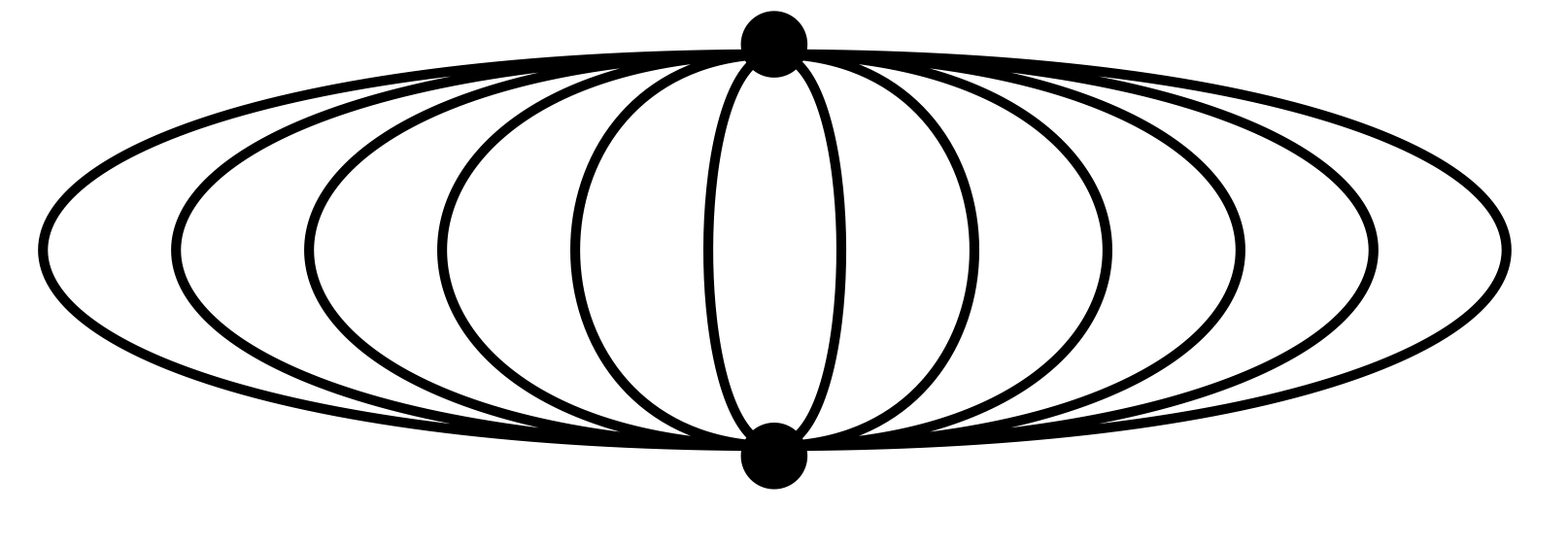}
	\vspace{-1.5em}
\end{center}
\caption{Reduction-graphs with lengths $\Gamma_{61}\backslash\mathcal{T}_{61}$ and $\Gamma_{61,+}\backslash\mathcal{T}_{61}$ for $X(3\cdot 61, 1)$}\label{ex2_2}
\vspace{-1.5em}
\end{figure}

\bibliography{LPM_bibliografia}

\Addresses

\end{document}